\documentclass[a4paper,11pt]{amsart}

\usepackage{amsmath}
\usepackage{amssymb}
\usepackage{amsthm}
\usepackage[abbrev,alphabetic]{amsrefs}
\usepackage{amscd}
\usepackage[dvipdfmx]{graphicx}

\usepackage[all,cmtip]{xy}

\usepackage{color}
\usepackage{xypic}

\title[Dual complex of log Fano pairs and its application]
{Dual complex of log Fano pairs and its application to Witt vector cohomology}

\author{Yusuke Nakamura}

\address{Graduate School of Mathematical Sciences, 
the University of Tokyo, 3-8-1 Komaba, Meguro-ku, Tokyo 153-8914, Japan.}

\email{nakamura@ms.u-tokyo.ac.jp}

\newtheorem{thm}{Theorem}[section]
\newtheorem{lem}[thm]{Lemma}

\newtheorem{prop}[thm]{Proposition}

\newtheorem{claim}[thm]{Claim}

\theoremstyle{definition}
\newtheorem{defi}[thm]{Definition}
\newtheorem{ex}[thm]{Example}

\theoremstyle{remark}
\newtheorem{rmk}[thm]{Remark}
\newtheorem*{ackn}{Acknowledgements}

\begin{document}
\begin{abstract}
We prove the contractibility of the dual complexes of weak log Fano pairs. 
As applications, we obtain a vanishing theorem of Witt vector cohomology of Ambro-Fujino type and 
a rational point formula in dimension three. 
\end{abstract}

\maketitle

\tableofcontents

\section{Introduction}
In the first half of this paper, we discuss the dual complex of Fano pairs. 
A dual complex is a combinatorial object which expresses how the components of $\Delta ^{=1}$ intersect for a dlt pair $(X, \Delta)$. 
In \cite{dFKX17}, they study the dual complex of a dlt modification of a pair $(X, \Delta)$, 
and show that the dual complex is independent of the choice of a dlt modification (minimal dlt blow-up) up to PL homeomorphism. 
In \cite{KX16} and \cite{Mau}, the dual complexes of log canonical dlt pairs $(X, \Delta)$ with $K_X + \Delta \sim _{\mathbb{Q}} 0$ are studied. 

Our main theorem is the contractibility of the dual complexes of weak Fano pairs (see \cite[22]{KX16} for a similar result). 

\begin{thm}[{$=$ Theorem \ref{thm:nefbig0}}]\label{thm:main1}
Let $(X, \Delta)$ be a projective pair over an algebraic closed field of characteristic zero. 
Assume that $-(K_X + \Delta)$ is nef and big. 
Then for any dlt blow-up $g: (Y, \Delta _Y) \to (X, \Delta)$, 
the dual complex $\mathcal{D}(\Delta _Y ^{\ge 1})$ is contractible, 
where we define $\Delta _Y$ by $K_Y + \Delta _{Y} = g^* (K_X + \Delta)$.
\end{thm}

\noindent
In this theorem, the coefficients of $\Delta$ might be larger than one contrary to the setting in \cite{dFKX17}. 
We prove that the dual complex is independent of the choice of a dlt blow-up 
(which is possibly not minimal) up to homotopy equivalence in our setting 
(Proposition \ref{prop:indep}, see also Remark \ref{rmk:independence}). 
The proof of Proposition \ref{prop:indep} depends on the weak factorization theorem \cite{AKMW02} and does not work 
in positive characteristic even in dimension three. 
Hence, we get the following weaker theorem in positive characteristic. 

\begin{thm}[{$=$ Theorem \ref{thm:nefbig}}]\label{thm:main1'}
Let $(X, \Delta)$ be a three dimensional projective pair over an algebraic closed field of characteristic larger than five. 
Assume that $-(K_X + \Delta)$ is nef and big. 
Then, there exists a dlt blow-up $g: (Y, \Delta _Y) \to (X, \Delta)$ such that 
the dual complex $\mathcal{D}(\Delta _Y ^{\ge 1})$ is contractible, 
where we define $\Delta _Y$ by $K_Y + \Delta _{Y} = g^* (K_X + \Delta)$.
\end{thm}

In the latter half of this paper, we discuss an application of the above result on the dual complex 
to the study on Witt vector cohomology in positive characteristic. 
In \cite{Esn03}, it is shown that the vanishing $H^i(X, W \mathcal{O}_{X, \mathbb{Q}}) = 0$ holds for $i > 0$ 
and a geometrically connected smooth Fano variety $X$ defined over an algebraic closed field $k$. 
This vanishing theorem is very impressive because it is not known whether 
$H^i(X, \mathcal{O}_{X}) = 0$ holds or not 
by the lack of the Kodaira vanishing theorem in positive characteristic. 
In \cite{GNT}, Esnault's result is generalized to klt pairs of dimension $3$ and 
in $\mathrm{char} (k) > 5$. 
In \cite{NT}, the result in \cite{GNT} is generalized as a vanishing theorem of Nadel type (see Theorem \ref{thm:WNV}). 
The following main theorem of this paper is another generalization of the result in \cite{GNT} 
(we note that the result in \cite{GNT} is Theorem \ref{thm:main2} with the additional restriction that 
the pair $(X, \Delta)$ is klt). 

\begin{thm}[{$=$ Theorem \ref{thm:WAFV}}]\label{thm:main2}
Let $k$ be a perfect field of characteristic $p > 5$. 
Let $(X, \Delta)$ be a three-dimensional projective $\mathbb{Q}$-factorial log canonical pair over $k$ with $-(K_X + \Delta)$ ample. 
Then $H^i(X,W \mathcal{O}_{X, \mathbb{Q}}) = 0$ holds for $i > 0$. 
\end{thm}

\noindent
By the vanishing theorem of Nadel type (Theorem \ref{thm:WNV}), the proof of Theorem \ref{thm:main2} is reduced to 
the topological study of the non-klt locus of the pair $(X, \Delta)$ (Proposition \ref{prop:normality} and Proposition \ref{prop:tree}). 
For the proof of Proposition \ref{prop:normality} and Proposition \ref{prop:tree}, 
we use the result on the dual complex (Theorem \ref{thm:main1'}) to obtain the topological information.

An important application of Witt vector cohomology is the rational point formula 
on varieties defined over a finite field (cf.\ \cite{Esn03, GNT, NT}). 
One of the motivation of the papers \cite{Esn03, GNT, NT} is to generalize the Ax-Katz theorem \cite{Ax64, Kat71}, 
which states that any hypersurface $H \subset \mathbb{P}^n$ of degree $d \le n$ defined over $\mathbb{F}_q$ 
has a rational point. 
Theorem \ref{thm:main2} suggests that the Ax-Katz theorem might be generalized to singular ambient spaces, 
and we actually obtain the following theorem in dimension three. 

\begin{thm}[{$=$ Theorem \ref{thm:RPF}}]\label{thm:main3}
Let $k$ be a finite field of characteristic $p > 5$. 
Let $(X, \Delta)$ be a geometrically connected three-dimensional projective $\mathbb{Q}$-factorial log canonical pair over $k$ with $-(K_X + \Delta)$ ample. 
Then the number of the $k$-rational points on the non-klt locus of $(X, \Delta)$ satisfies 
\[
\# \mathrm{Nklt}(X, \Delta) (k) \equiv 1 \mod {|k|}.
\]
In particular, there exists a $k$-rational point on $\mathrm{Nklt}(X, \Delta)$. 
\end{thm}

\noindent
If $X$ is klt and $\Delta$ is a reduced divisor, then 
$\mathrm{Nklt}(X, \Delta) = \mathrm{Supp} (\Delta)$ holds and Theorem \ref{thm:main3} claims that 
there exists a $k$-rational point on $\mathrm{Supp} (\Delta)$. 
This formulation can be seen as a generalization of the Ax-Katz theorem from the view point of birational geometry. 

\begin{ackn} 
We would like to thank Professors Yoshinori Gongyo and Hiromu Tanaka for the discussions and Mirko Mauri for useful comments and suggestions. 
The author is partially supported by the Grant-in-Aid for Young Scientists
(KAKENHI No.\ 18K13384).
\end{ackn}

\section{Preliminaries}\label{section:prelimi}

\subsection{Notation}\label{subsection:notation}

\begin{itemize}
\item We basically follow the notations and the terminologies in \cite{Har77} and \cite{Kol13}.

\item 
For a field $k$, we say that $X$ is a \textit{variety over} $k$ if 
$X$ is an integral separated scheme of finite type over $k$. 

\item 
A \textit{sub log pair} $(X, \Delta)$ over a filed $k$ consists of a normal variety $X$ over $k$ and 
an $\mathbb{R}$-divisor $\Delta$ such that $K_X + \Delta$ is $\mathbb{R}$-Cartier. 
A sub log pair is called \textit{log pair} if $\Delta$ is effective. 
Note that the coefficient of $\Delta$ may be larger than one in this definition. 

\item 
Let $\Delta = \sum r_i D_i$ be an $\mathbb{R}$-divisor where $D_i$ are distinct prime divisors. 
We define 
$\Delta ^{\ge 1} := \sum_{r_i \ge 1} r_i  D_i$ and 
$\Delta ^{\wedge 1} := \sum r_i ' D_i$ where $r_i ' := \min \{ r_i, 1  \}$. 
We also define $\Delta ^{> 1}, \Delta ^{< 1}$, and $\Delta ^{= 1}$ similarly. 

\item Let $X$ be a variety over $k$. 
We denote by $(\star)$ the condition that $k$ and $X$ satisfy one of the following conditions:
\begin{itemize}
\item[(i)] $\mathrm{ch} (k) = 0$, or
\item[(ii)] $\mathrm{ch} (k) > 5$ and $\dim X = 3$. 
\end{itemize}
This condition is necessary for running a certain MMP appeared in this paper \cite{BCHM10, HX15, Bir16, BW17, Wal18, HNT}. 
\end{itemize}

\subsection{Results on minimal model program}
In this subsection we review results on minimal model program. 
In this subsection, $k$ is an algebraic closed field. 

First, we review the definition of singularities of log pairs. 
In this paper, we treat the following definitions only under the condition $(\star)$, 
because we do not know whether the definitions perform well also in $\dim X > 3$ and $\mathrm{ch} (k) >0$. 

\begin{defi}
\begin{enumerate}
\item 
Let $(X, \Delta)$ be a log pair over $k$. 
For a proper birational $k$-morphism $f: X' \to X$ from a normal variety $X'$ 
and a prime divisor $E$ on $X'$, the \textit{log discrepancy} of $(X, \Delta)$ 
at $E$ is defined as
\[
a_E (X, \Delta) := 1 + \mathrm{coeff}_E (K_{X'} - f^* (K_X + \Delta)). 
\]

\item 
A log pair $(X, \Delta)$ is called \textit{klt} (resp.\ \textit{log canonical}) if 
$a_E (X, \Delta) > 0$ (resp. $\ge 0$) for any prime divisor $E$ over $X$. 

\item 
A log pair $(X, \Delta)$ is called \textit{dlt} if the coefficients of $\Delta$ are at most one and 
there exists a log resolution $g: Y \to X$ of the pair $(X, \Delta)$ such that 
$a_E (X, \Delta) > 0$ holds for any $g$-exceptional prime divisor $E$ on $Y$. 

\item 
Let $(X, \Delta)$ be a dlt pair and let $\Delta ^{=1} = \sum _{i \in I} E_i$ be the irreducible decomposition. 
For any non-empty subset $J \subset I$, 
an connected component of $\bigcap _{i \in J} E_i$ is called a \textit{stratum} of $\Delta ^{=1}$. 
\end{enumerate}
\end{defi}

\begin{rmk}
The above definition (3) is equivalent to the definition in \cite[Definition 2.37]{KM98}: 
\begin{itemize}
\item The coefficients of $\Delta$ are at most one. 
Moreover, there exists an open subset $U \subset X$ such that $(U, \Delta |_U)$ is log smooth and no non-klt center of 
$(X, \Delta)$ is contained in $X \setminus U$. 
\end{itemize}
This equivalence is shown in \cite{Sza94} in characteristic zero. 
The equivalence is also true in positive characteristic (in dimension three) 
since the Szab\'{o}'s resolution lemma (\cite[Lemma 2.3.19]{Fuj17}) also holds by \cite[Proposition 4.1]{CP08} 
(see also \cite[Proposition 2.3.20]{Fuj17} and \cite[2.4, 2.5]{Bir16}). 
Hence, even in our definition, being dlt is preserved under the MMP (\cite[Corollary 3.44]{KM98}). 
\end{rmk}

The following proposition is necessary for defining the dual complexes of dlt pairs. 

\begin{prop}[{\cite[Theorem 4.16]{Kol13}, \cite[Section 3.9]{Fuj07}, \cite[Proposition 1]{DH16}}]\label{prop:dltstrata}
Let $(X, \Delta)$ be a $\mathbb{Q}$-factorial dlt pair over $k$ and 
$\Delta ^{=1} = \sum _{i \in I} E_i$ be the irreducible decomposition. 
We assume the condition $(\star)$ (defined in Subsection \ref{subsection:notation}). 
Then the following hold. 
\begin{enumerate}
\item Let $J \subset I$ be a subset. If $\bigcap _{i \in J} E_i \not= \emptyset$, 
then each connected component of $\bigcap _{i \in J} E_i$ is normal (hence irreducible) and 
has codimension $\# J$. 

\item Let $J \subset I$ be a subset, and let $j \in J$. 
Then each connected component of $\bigcap _{i \in J} E_i$
is contained in the unique connected component of $\bigcap _{i \in J \setminus \{ j \}} E_i$. 
\end{enumerate}
\end{prop}

\begin{proof}
See \cite[Theorem 4.16]{Kol13}. The assertion that each connected component of $\bigcap _{i \in J} E_i$ is irreducible 
is not explicitly written in \cite[Theorem 4.16]{Kol13}. 
However it follows from the fact that the intersection of any two log canonical centers is also a union of log canonical centers 
(cf.\ \cite[Theorem 9.1]{Fuj11}, \cite[Lemma 1]{DH16}). 

(2) is trivial. 
\end{proof}

In this paper, we will use the terminology ``dlt blow-up" in the following sense. 
\begin{defi}
Let $(X, \Delta)$ be a log pair over $k$ and let $g:Y \to X$ be a projective birational $k$-morphism. 
We call $g$ a \textit{dlt blow-up} of $(X, \Delta)$ if the following conditions hold:
\begin{itemize}
\item[(1)] $a_E(X, \Delta) \le 0$ holds for any $g$-exceptional prime divisor $E$. 
\item[(2)] $(Y, \Delta _Y ^{\wedge 1})$ is a $\mathbb{Q}$-factorial dlt pair, where 
$\Delta _Y$ is the $\mathbb{R}$-divisor defined by $K_Y + \Delta _Y = g^*(K_X + \Delta)$. 
\end{itemize}
\end{defi}

\begin{thm}\label{thm:dltmodif}
Let $(X, \Delta)$ be a log pair over $k$ with the condition $(\star)$. 
Then a dlt blow-up of $(X, \Delta)$ exists. 
Further, we can take a dlt blow-up $g: V \to X$ with the following additional condition: 
\begin{itemize}
\item[(3)] $g^{-1}(\mathrm{Nklt} (X, \Delta)) = \mathrm{Nklt} (V, \Delta _V)$ holds. 
\end{itemize}
\end{thm}
\begin{proof}
By Step 2 in the proof of \cite[Proposition 3.5]{HNT}, 
in order to show the existence of a dlt blow-up with the condition (3), 
it is sufficient to show the existence of a usual dlt blow-up (that is with the only conditions (1) and (2)). 
For the existence of a dlt blow-up, the same proof as in \cite[Theorem 10.4]{Fuj11} works as follows. 

Let $f: Y \to X$ be a log resolution of $(X, \Delta)$. 
Let $F = \sum F_i$ be the sum of the $f$-exceptional divisors $F_i$ with $a_{F_i} (X, \Delta) \le 0$, and 
let $G = \sum G_i$ be the sum of the $f$-exceptional divisors $G_i$ with $a_{G_i} (X, \Delta) > 0$. 
Let $\widetilde{\Delta}$ be the strict transform of $\Delta$ on $Y$. 
We may assume that there exists an effective $\mathbb{R}$-divisor $H$ on $Y$ such that 
$\mathrm{Supp}\, H = \mathrm{Supp}\, (F+G)$ and that $-H$ is $f$-ample. 
We set an $\mathbb{R}$-divisor $\Omega$ as 
\[
\Omega = \widetilde{\Delta}^{\wedge 1} + F + (1 - \epsilon) G - \delta H 
\]
for sufficiently small $\epsilon, \delta > 0$. 
Since $-H$ is $f$-ample, there exists an effective ample $\mathbb{R}$-divisor $A$ on $Y$ such that 
$- \delta H \sim _{/X,\, \mathbb{R}} A$. 
We set 
\[
\overline{\Omega} = \widetilde{\Delta}^{\wedge 1} + F + (1 - \epsilon) G + A. 
\]
We may assume that $(Y, \overline{\Omega})$ is dlt. 
Note that $K_Y + \Omega \sim _{/X,\, \mathbb{R}} K_Y + \overline{\Omega}$. 

Since $A$ is ample, there exists an $\mathbb{R}$-divisor $\overline{\Omega}'$ such that 
$\overline{\Omega}' \sim _{\mathbb{R}} \overline{\Omega}$ and $(Y, \overline{\Omega}')$ is klt. 
Hence, we may run a $(K_Y + \Omega)$-MMP over $X$ and it terminates. 
Let $Y'$ be the end result and let $h: Y' \to X$ be the induced morphism. 
We shall show that $h: Y' \to X$ is a dlt blow-up of $(X, \Delta)$. 
We have 
\begin{align*}
K_Y + \Omega 
& \sim _{/X,\, \mathbb{R}} K_Y + \Omega - f^*(K_X + \Delta) \\
&= - (\widetilde{\Delta} -  \widetilde{\Delta}^{\wedge 1}) + \sum a_i F_i + \sum b_i G_i - \epsilon G - \delta H, 
\end{align*}
where $a_i = a_{F_i} (X, \Delta) \le 0$ and $b_i = a_{G_i} (X, \Delta) >0$. 
Since $b_i > 0$ and $\epsilon$ and $\delta$ are sufficiently small, 
it follows that 
\[
\mathrm{coeff}_{G_j} \left( \sum a_i F_i + \sum b_i G_i - \epsilon G - \delta H \right) > 0
\] 
for each $G_j$. 
Hence by the negativity lemma, all the divisors $G_i$'s are contracted in this MMP. 
Therefore $a_E (X, \Delta) \le 0$ holds for any $h$-exceptional prime divisor $E$. 
Since the $(K_Y + \Omega)$-MMP is also a $(K_Y + \overline{\Omega})$-MMP, 
the pair $(Y', \overline{\Omega} _{Y'})$ is still dlt where $\overline{\Omega} _{Y'}$ is the push forward of $\overline{\Omega}$. 
Define $\Delta _{Y'}$ by $K_{Y'} + \Delta _{Y'} = h^* (K_X + \Delta)$. 
Then $\Delta _{Y'} ^{\wedge 1}$ is the push forward of $\widetilde{\Delta}^{\wedge 1} + F$ on $Y'$, 
and hence $0 \le \Delta _{Y'} ^{\wedge 1} \le \overline{\Omega} _{Y'}$ holds. 
Therefore $(Y', \Delta _{Y'} ^{\wedge 1})$ is also dlt. 
We have proved that $g$ is a dlt blow-up of $(X, \Delta)$. 
\end{proof}

\begin{rmk}\label{rmk:dltbup}
\begin{enumerate} 
\item When $X$ is $\mathbb{Q}$-factorial, any dlt blow-up of $(X, \Delta)$ satisfies condition (3) in 
Theorem \ref{thm:dltmodif}. 
\item 
If $V \to X$ is a log resolution of $(X, \Delta)$, 
then we can construct (by the proof above) a dlt blow-up $Y \to X$ of $(X, \Delta)$ such that 
the induced birational map $Y \dasharrow V$ does not contract any divisor on $Y$. 
\end{enumerate}
\end{rmk}

\subsection{Dual complexes}\label{subsection:dual_cpx}
In this subsection, we explain how to define a CW complex from a dlt pair, and we also prove an invariant property (Proposition \ref{prop:indep}). 

First, we briefly review the notion of $\Delta$-complexes following \cite{Hat02}. 

\begin{defi}
\begin{enumerate}
\item Let $X = \cup _{\varphi _{\alpha}} F_{\alpha}$ be a CW complex with the attaching maps $\varphi _{\alpha} : F_{\alpha} \to X$. 
We call $X$ a \textit{$\Delta$-complex} when each cell $F_{\alpha}$ is a simplex and 
the restriction of $\varphi _{\alpha}$ to each face of $F_{\alpha}$ is equal to
the attaching map $\varphi _{\beta} : F_{\beta} \to X$ for some $\beta$. 

\item
A $\Delta$-complex $X$ is called \textit{regular} if the attaching maps are injective, 
or equivalently, if every $d$-cell in $X$ has $d+1$ distinct vertices. 

\item
A regular $\Delta$-complex $X$ is called a \textit{simplicial complex} 
if the intersection of any two cells in $X$ is a face of both cells, or equivalently, 
every $k+1$ vertices in $X$ is incident to at most one $k$-cell. 
\end{enumerate}
\end{defi}

For a dlt pair $(X, \Delta)$, we define the dual complex $\mathcal{D}(\Delta ^{=1})$. 
\begin{defi}
\begin{enumerate}
\item Let $(X, \Delta)$ be a $\mathbb{Q}$-factorial dlt pair over $k$ with the condition $(\star)$ and 
let $\Delta ^{=1} = \sum _{i \in I} E_i$ be the irreducible decomposition. 
Then the \textit{dual complex} $\mathcal{D}(\Delta ^{=1})$ is a CW complex obtained as follows. 
The vertices of $\mathcal{D}(\Delta ^{=1})$ are the set of $\{ E_i \} _{i \in I}$. 
To each $k$-codimensional stratum $S$ of $\Delta ^{=1}$ we associate a $k$-dimensional cell. 
The attaching map is uniquely defined by Proposition \ref{prop:dltstrata} (2). 

\item Let $(X, \Delta)$ be a $\mathbb{Q}$-factorial pair over $k$ with the condition $(\star)$. 
Suppose that $(X, \Delta^{\wedge 1})$ is dlt. Then we define 
$\mathcal{D} (\Delta ^{\ge 1}) = \mathcal{D} ((\Delta^{\wedge 1}) ^{= 1})$
\end{enumerate}
\end{defi}

\begin{prop}\label{prop:regular}
Let $(X, \Delta)$ be a $\mathbb{Q}$-factorial dlt pair over $k$ with the condition $(\star)$. 
Then the dual complex $\mathcal{D}(\Delta ^{=1})$ is a regular $\Delta$-complex. 
\end{prop}
\begin{proof}
The assertion follows from Proposition \ref{prop:dltstrata} (1). See \cite{dFKX17} for more detail. 
\end{proof}

The following theorem from \cite{dFKX17} says that the dual complex is preserved under a certain MMP up to simple-homotopy equivalence. 

\begin{thm}[{\cite[Proposition 19]{dFKX17}}]\label{thm:inv_mmp}
Let $(X, \Delta)$ be a $\mathbb{Q}$-factorial dlt pair over $k$ with condition $(\star)$ and 
let $f:X \dasharrow Y$ be a divisorial contraction or 
flip corresponding to a $(K_X + \Delta)$-negative extremal ray $R$. 
Assume that there is a prime divisor $D_0 \subset \Delta ^{=1}$ such that $D_0 \cdot R >0$. 
Then $\mathcal{D}(\Delta ^{=1})$ collapses to $\mathcal{D}(\Delta _Y ^{=1})$ where we set $\Delta _Y := f_* \Delta$. 
In particular $\mathcal{D}(\Delta ^{= 1})$ and $\mathcal{D}(\Delta _Y ^{=1})$ are simple-homotopy equivalent. 
\end{thm}

\begin{lem}\label{lem:szabo}
Let $(X, \Delta)$ be a $\mathbb{Q}$-factorial pair over $k$ with the condition $(\star)$. 
Suppose that $(X, \Delta ^{\wedge 1})$ is dlt. 
Let $f: Y \to X$ be a log resolution of $(X, \Delta ^{\wedge 1})$ such that 
$a_E (X, \Delta ^{\wedge 1}) > 0$ holds for any $f$-exceptional divisor $E$. 
Let $\Delta _Y$ be the (not necessarily effective) $\mathbb{R}$-divisor defined by $K_Y + \Delta _Y = f^*(K_X + \Delta)$. 
Then the dual complexes $\mathcal{D}(\Delta ^{\ge 1})$ and $\mathcal{D}(\Delta _Y ^{\ge 1})$ are simple-homotopy equivalent. 
\end{lem}
\begin{proof}
Let $F = \sum F_i$ be the sum of the $f$-exceptional divisors $F_i$ with $a_{F_i} (X, \Delta) \le 0$, and 
let $G = \sum G_i$ be the sum of the $f$-exceptional divisors $G_i$ with $a_{G_i} (X, \Delta) > 0$. 
Let $\widetilde{\Delta}$ be the strict transform of $\Delta$ on $Y$. 
Let $H$ be an effective $\mathbb{R}$-divisor on $Y$ such that 
$\mathrm{Supp}\, H = \mathrm{Supp}\, (F+G)$ and that $-H$ is $f$-ample. 
We set an $\mathbb{R}$-divisor $\Omega$ as 
\[
\Omega = \widetilde{\Delta}^{\wedge 1} + F + (1 - \epsilon) G - \delta H 
\]
for sufficiently small $\epsilon, \delta > 0$. 
Since $-H$ is $f$-ample, there exists an effective ample $\mathbb{R}$-divisor $A$ on $Y$ such that 
$- \delta H \sim _{/X,\, \mathbb{R}} A$. 
We set 
\[
\overline{\Omega} = \widetilde{\Delta}^{\wedge 1} + F + (1 - \epsilon) G + A. 
\]
We may assume that $(Y, \overline{\Omega})$ is dlt. 
Note that $K_Y + \Omega \sim _{/X,\, \mathbb{R}} K_Y + \overline{\Omega}$ and $\mathcal{D}(\overline{\Omega} ^{=1}) = \mathcal{D}(\Delta _Y ^{\ge 1})$. 

Since $A$ is ample, there exists an $\mathbb{R}$-divisor $\overline{\Omega}'$ such that 
$\overline{\Omega}' \sim _{\mathbb{R}} \overline{\Omega}$ and $(Y, \overline{\Omega}')$ is klt. 
Hence, we may run a $(K_Y + \Omega)$-MMP over $X$ and it terminates. 
First, we prove that this MMP ends with $X$. 
\begin{align*}
K_Y + \Omega 
& \sim _{/X,\, \mathbb{R}} K_Y + \Omega - f^*(K_X + \Delta^{\wedge 1}) \\
&= \sum a_i F_i + \sum b_i G_i - \epsilon G - \delta H, 
\end{align*}
where $a_i = a_{F_i} (X, \Delta ^{\wedge 1})$ and $b_i = a_{G_i} (X, \Delta ^{\wedge 1})$. 
Since $a_i, b_i > 0$ and $\epsilon$ and $\delta$ are sufficiently small, 
it follows that 
\[
\mathrm{coeff}_E \left( \sum a_i F_i + \sum b_i G_i - \epsilon G - \delta H \right) > 0
\] 
for any $f$-exceptional divisor $E$. 
Hence by the negativity lemma, all the divisors $F_i$'s and $G_i$'s are contracted in this MMP. 
Since $X$ is $\mathbb{Q}$-factorial, this MMP ends with $X$. 

Let $Y_j \dasharrow Y_{j+1}$ be the step of the $(K_Y + \Omega)$-MMP over $X$, and 
let $R$ be the corresponding extremal ray, and $Y_j \to Z_j$ be its contraction. 
For a divisor $D$ on $Y$, we denote $D_{Y_j}$ the strict transform of $D$ on $Y_j$. 
We also denote $g: Y_j \to X$ the induced morphism. 
Then, 
\begin{align*}
&  K_{Y_j} + \Omega _{Y_j} \\
\sim & _{/X,\, \mathbb{R}} K_{Y_j} + \Omega _{Y_j} - g^*(K_X + \Delta) \\
= & - \left( \widetilde{\Delta}_{Y_j} - (\widetilde{\Delta} _{Y_j}) ^{\wedge 1} \right) + 
\sum a' _i F_{i, Y_j} + \sum b' _i G_{i, Y_j} - \epsilon G_{Y_j} - \delta H_{Y_j}, 
\end{align*}
where $a' _i = a_{F_i} (X, \Delta)$ and $b' _i = a_{G_i} (X, \Delta)$. 
Since $a' _i \le 0$, it follows that 
\[
\mathrm{coeff}_{F_{i, Y_j}} \left( \sum a' _i F_{i, Y_j} + \sum b' _i G_{i, Y_j} - \epsilon G_{Y_j} - \delta H_{Y_j} \right) < 0. 
\]
Since $b' _i > 0$ and $\epsilon$ and $\delta$ are sufficiently small, it follows that 
\[
\mathrm{coeff}_{G_{i, Y_j}} \left( \sum a' _i F_{i, Y_j} + \sum b' _i G_{i, Y_j} - \epsilon G_{Y_j} - \delta H_{Y_j} \right) > 0. 
\]
Since $(K_{Y_j} + \Omega_{Y_j}) \cdot R < 0$, at least one of the following conditions hold: 
\begin{enumerate}
\item $D \cdot R > 0$ holds for some component $D \subset \mathrm{Supp} (\widetilde{\Delta} _{Y_j} ^{>1})$. 
\item $F_{i, Y_j} \cdot R > 0$ for some $i$. 
\item $G_{i, Y_j} \cdot R < 0$ for some $i$.
\end{enumerate}
Here, we have $K_{Y_j} + \Omega _{Y_j} \sim _{/X,\, \mathbb{R}} K_{Y_j} + \overline{\Omega} _{Y_j}$ and 
that a component of $\overline{\Omega} _{Y_j} ^{=1}$ is one of $F_{i, Y_j}$'s or 
a component of $(\widetilde{\Delta} _{Y_j}) ^{\ge 1}$. 
Hence in the case (1) or (2), by Theorem \ref{thm:inv_mmp}, 
the dual complexes $\mathcal{D}((\overline{\Omega} _{Y_j}) ^{=1})$ and 
$\mathcal{D}((\overline{\Omega} _{Y_{j+1}}) ^{=1})$ are simple-homotopy equivalent. 
In the case (3), the exceptional locus $L$ of $Y_j \to Z_j$ is contained in $\mathrm{Supp} (G_{i, Y_j})$. 
Since $(Y_j, \overline{\Omega} _{Y_j})$ is dlt, any stratum of $(\overline{\Omega} _{Y_j}) ^{=1}$ 
is not contained in $\mathrm{Supp} (G_{i, Y_j})$ and neither in $L$. 
Hence in the case (3), by \cite[Lemma 16]{dFKX17}, it follows that 
$\mathcal{D}((\overline{\Omega} _{Y_j}) ^{=1}) = \mathcal{D}((\overline{\Omega} _{Y_{j+1}}) ^{=1})$. 

Hence by induction on $j$, the dual complexes $\mathcal{D}(\overline{\Omega} ^{=1})$ and 
$\mathcal{D}((\overline{\Omega} _X) ^{=1})$ are simple-homotopy equivalent. 
Since $\mathcal{D}(\overline{\Omega} ^{=1}) = \mathcal{D}(\Delta _Y ^{\ge 1})$ and 
$\mathcal{D}((\overline{\Omega} _X) ^{=1}) = \mathcal{D}(\Delta ^{\ge 1})$, 
it follows that $\mathcal{D}(\Delta _Y ^{\ge 1})$ and $\mathcal{D}(\Delta ^{\ge 1})$ are homopoty equivalent. 
\end{proof}

\begin{lem}\label{lem:bup}
Let $(X, \Delta)$ be a sub log pair over $k$ with the condition $(\star)$ such that $(X, \mathrm{Supp}\ \Delta)$ is log smooth. 
Let $Z$ be a smooth irreducible subvariety of $X$ which has only simple normal 
crossing with $\mathrm{Supp}\, \Delta$. 
Let $f : Y \to X$ be the blow up along $Z$, and 
let $\Delta _Y$ be the $\mathbb{R}$-divisor defined by $K_Y + \Delta _Y = f^*(K_X + \Delta)$.
Then the dual complexes $\mathcal{D}(\Delta ^{\ge 1})$ and $\mathcal{D}(\Delta _Y ^{\ge 1})$ are simple-homotopy equivalent. 
\end{lem}
\begin{proof}
Set $F = (\Delta ^{\ge 1})^{\wedge 1}$ and $F_Y = (\Delta_Y ^{\ge 1})^{\wedge 1}$. 
Let $E$ be the $f$-exceptional divisor and $\widetilde{F}$ be the strict transform of $F$ on $Y$. 
We divide the case as follows: 
\begin{enumerate}
\item $Z$ is a stratum of $F$. 
\item $Z \subset \mathrm{Supp}\, F$ but $Z$ is not a stratum of $F$.
\item $Z \not \subset \mathrm{Supp}\, F$ (in particular, $Z$ is not a stratum of $F$).
\end{enumerate}

Suppose (1). Then $F_Y = \widetilde{F} + E$ holds. 
On the other hand, $\mathcal{D}(\widetilde{F} + E)$ and $\mathcal{D}(F)$ are PL homeomorphic each other by \cite[9]{dFKX17}. 

Suppose (2). Then $F_Y = \widetilde{F}$ or $F_Y = \widetilde{F} + E$ holds. 
On the other hand, $\mathcal{D}(\widetilde{F}) = \mathcal{D}(F)$ holds and 
$\mathcal{D}(\widetilde{F} + E)$ and $\mathcal{D}(F)$ are simple-homotopy equivalent by \cite[9]{dFKX17}. 

Suppose (3). Then $F_Y = \widetilde{F}$ holds and $\mathcal{D}(\widetilde{F}) = \mathcal{D}(F)$. 
\end{proof}

\begin{prop}\label{prop:indep}
Let $(X, \Delta)$ be a pair over $k$ with condition (i) in $(\star)$ and 
let $f_1: Y_1 \to X$ and $f_2: Y_2 \to X$ be two dlt blow-ups of $(X, \Delta)$. 
Define $\mathbb{R}$-divisors $\Delta _{Y_i}$ by 
$K_{Y_i} + \Delta _{Y_i} = f_i ^* (K_X + \Delta)$. 
Then the dual complexes $\mathcal{D}(\Delta _{Y_1} ^{\ge 1})$ and 
$\mathcal{D}(\Delta _{Y_2} ^{\ge 1})$ are simple-homotopy equivalent. 
\end{prop}
\begin{proof}
The same proof of \cite[Proposition 11]{dFKX17} works. 
For the reader's convenience, we give a sketch of proof. 

By definition of dlt pairs, we can take a log resolution $g_i: W_i \to Y_i$ of $(Y_i, \Delta _{Y_i})$ such that 
$a_E (Y_i, \Delta _{Y_i}^{\wedge 1}) > 0$ holds for any $g_i$-exceptional divisor $E$. 
Define $\mathbb{R}$-divisors $\Delta _{W_i}$ on $W_i$ by $K_{W_i} + \Delta _{W_i} = g_i ^* (K_{Y_i} + \Delta _{Y_i})$. 
By Lemma \ref{lem:szabo}, the dual complexes $\mathcal{D} (\Delta _{Y_i} ^{\ge 1})$ and 
$\mathcal{D} (\Delta _{W_i} ^{\ge 1})$ are simple-homotopy equivalent. 
By the weak factorization theorem \cite{AKMW02}, the pairs $(W_1, \Delta _{W_1})$ and 
$(W_2, \Delta _{W_2})$ can be connected by a sequence of blow up as in Lemma \ref{lem:bup}. 
Hence $\mathcal{D}(\Delta _{W_1} ^{\ge 1})$ and $\mathcal{D}(\Delta _{W_2} ^{\ge 1})$ are simple-homotopy equivalent. 
\end{proof}

\begin{rmk}\label{rmk:independence}
\begin{enumerate}
\item 
In the proof of this proposition, the weak factorization theorem \cite{AKMW02} is used. 
So, the proof does not work in positive characteristic even in dimension three. 

\item If $(X, \Delta)$ is log canonical in this proposition, then
$\mathcal{D}(\Delta _{Y_1} ^{\ge 1})$ and 
$\mathcal{D}(\Delta _{Y_2} ^{\ge 1})$ are PL homeomorphic each other (\cite[Proposition 11]{dFKX17}). 
\end{enumerate}
\end{rmk}

\subsection{Results on the Witt vector cohomologies}
For the definition of the Witt vector cohomology and its basic properties, we refer to \cite{GNT} and \cite{CR12}. 
The following vanishing theorem of Nadel type will be used in this paper. 

\begin{thm}[{\cite[Theorem 4.10]{NT}}]\label{thm:WNV}
Let $(X, \Delta)$ be a projective log pair over a perfect field $k$ with condition (ii) in $(\star)$. 
Then 
\[
H^i(X, WI_{\mathrm{Nklt}(X, \Delta), \mathbb{Q}}) = 0
\]
holds for $i >0$, where $\mathrm{Nklt}(X, \Delta)$ denotes the reduced closed subscheme of $X$ consisting of the non-klt points of $(X, \Delta)$ 
and $I_{\mathrm{Nklt}(X, \Delta)}$ is the coherent ideal sheaf on $X$ corresponding to $\mathrm{Nklt}(X, \Delta)$. 
\end{thm}

\section{Dual complex of weak Fano varieties}

\subsection{Dual complex of dlt pairs with a Mori fiber space structure}

\begin{lem}\label{lem:MFS2}
Let $(X, \Delta)$ be a $\mathbb{Q}$-factorial dlt pair over $k$ with condition $(\star)$ and 
let $f: X \to Z$ be a projective surjective $k$-morphism to a quasi-projective $k$-scheme $Z$ such that 
\begin{enumerate}
\item 
$\dim X >\dim Z$, 
\item 
$f$ has connected fibers, 
\item 
$-(K_X+\Delta)$ is $f$-ample, and 
\item 
there exists an irreducible component $D_0$ of $\Delta ^{=1}$ 
such that $D_0$ is $f$-ample. 
\end{enumerate}
Then $\mathcal{D}(\Delta ^{=1})$ is contractible.
\end{lem}

\begin{proof}
Let $g:D_0 \to Z$ be the induced morphism. 
Let $\Delta ^{=1} = \sum_{i=0}^m D_i$ be the irreducible decomposition. 

What we want to show is the following: 
\begin{itemize}
\item[(1)] Any stratum $S$ of $\sum_{i=1}^m D_i$ intersects with $D_0$, and
\item[(2)] $S \cap D_0$ is connected. 
\end{itemize}
These conditions imply that the dual complex of $\sum_{i=0}^m D_i$ is the cone of the dual complex of 
$\sum_{i=1}^m D_i$ with the vertex $D_0$. 
Therefore $\mathcal{D}(\Delta ^{=1})$ is contractible. 
We show (1) and (2) by induction on $\dim X$. 

First we prove the following claim. 
\begin{claim}\label{claim:majiwaru}
For any $i \in \{1, \ldots, m\}$, 
it holds that 
\begin{enumerate}
\item[(a)] $\dim f(D_i) < \dim D_i$, and
\item[(b)] $f(D_i) =f(D_i \cap D_0)$ holds, in particular $D_i \cap D_0 \not= \emptyset$. 
\item[(c)] $D_i \cap D_0$ is connected and irreducible. 
\end{enumerate}
\end{claim}
\begin{proof}
We prove (a). 
Since the assertion is clear if $\dim Z \leq \dim X - 2$, 
we may assume that $\dim Z = \dim X - 1$. 

Suppose that $f(D_i)=Z$ holds for some $i \in \{1, \dots, m\}$. 
For a general fiber $F$ of $f$, $\dim F = 1$ holds and $(D_i \cup D_0) \cap F$ is not connected. 
This contradicts the Koll{\'a}r-Shokurov connectedness lemma (cf.\ \cite[Theorem 1.2]{NT}). 
Therefore, $D_i$ does not dominate $Z$ for any $i \in \{1, \dots, m\}$. 

We prove (b). 
Let $x \in f(D_i)$ be a closed point. 
Since $\dim f(D_i) < \dim D_i$, 
there exists a curve $C$ on $X$ contained in $D_i \cap f^{-1}(x)$. 
Since $D_0$ is ample over $Z$, 
the contracted curve $C$ intersects with $D_0$. 
This implies $x \in f(D_i \cap D_0)$. 
Thus we get $f(D_i) = f(D_i \cap D_0)$. 

We prove (c). Suppose that $D_i \cap D_0$ is not connected. 
Let $S$ be a connected component of $D_i \cap D_0$ which satisfies $f(S) = f(D_i \cap D_0)$. 
Let $G$ be another connected component of $D_i \cap D_0$. 
Then for any closed point $x \in f(G)$, $D_i \cap D_0$ is not connected over $x$. 
However, this contradicts the Koll{\'a}r-Shokurov connectedness lemma. 
Therefore $D_i \cap D_0$ is connected. 
Then the irreducibility follows from Proposition \ref{prop:dltstrata} (1). 
\end{proof}

Let $S$ be a stratum of $\sum_{i=1}^m D_i$. 
Then $S$ is a connected component of $\bigcap _{i \in I} D_i$ for some 
$I \subset \{ 1, \ldots ,m \}$. 
We may assume that $1 \in I$ possibly changing the indices. 

Let $C:= f(D_1)$, and 
let $D_1 \xrightarrow{f'} C' \xrightarrow{s} C$ be the Stain factorisation of $D_1 \to C$. 
Let $\Delta _{D_1}$ be the effective $\mathbb{R}$-divisor on $D_1$ defined by adjunction 
$(K_X+\Delta)|_{D_1}=K_{D_1}+\Delta _{D_1}$. 
Then the following properties hold. 
\begin{enumerate}
\item[(d)] $(D_1, \Delta _{D_1})$ is dlt.
\item[(e)] $-(K_{D_1} + \Delta _{D_1})$ is $f'$-ample. 
\item[(f)] $- D_0|_{D_1}$ is $f'$-ample. 
\end{enumerate}

We prove (1) and (2) by induction on $\# I$. 
If $\# I = 1$, then (1) and (2) follow from  Claim \ref{claim:majiwaru}.
Suppose $\# I \ge 2$. Then $S$ is also a stratum of $\Delta _{D_1} ^{=1}$. 
Hence (1) and (2) holds by induction on the dimension. 
\end{proof}

\begin{lem}\label{lem:MFS} 
Let $(X, \Delta)$ be a $\mathbb{Q}$-factorial dlt pair over $k$ with the condition $(\star)$ and 
let $f: X \to Z$ be a $(K_X+\Delta)$-Mori fiber space to a quasi-projective $k$-variety $Z$. 
Suppose that $f(\mathrm{Supp}\, \Delta ^{=1})=Z$. 
Then $\mathcal{D}(\Delta ^{=1})$ is contractible. 
\end{lem}

\begin{proof}
Since $f(\mathrm{Supp}\, \Delta ^{=1})=Z$, 
some irreducible component $D_0$ of $\Delta ^{=1}$ satisfies $f(D_0)=Z$. 
Since $\rho(X/Z)=1$, it follows that $D_0$ is $f$-ample. 
Hence the assertion follows from Lemma \ref{lem:MFS2}. 
\end{proof}

\subsection{Dual complex of weak Fano varieties}

\begin{thm}\label{thm:scc2}
Let $(X, \Omega)$ be a $\mathbb{Q}$-factorial projective log pair. 
Assume that 
\begin{itemize}
\item[(1)] $(X, \Omega ^{\wedge 1})$ is dlt but not klt. 
\item[(2)] $K_X + \Omega \sim _{\mathbb{R}} 0$. 
\item[(3)] $\mathrm{Supp}\, \Omega ^{>1} = \mathrm{Supp}\, \Omega ^{\ge 1}$. 
\end{itemize}
Then, the dual complex $\mathcal{D}(\Omega ^{\ge 1})$ is contractible. 
\end{thm}

\begin{proof}
Since $(X, \Omega ^{\wedge 1})$ is not klt, it follows that 
\[
\mathrm{Supp}\, \Omega ^{>1} = \mathrm{Supp}\, \Omega ^{\ge 1} \not = \emptyset. 
\]
Therefore, $K_X+\Omega^{\wedge 1} \sim _{\mathbb{R}} - (\Omega-\Omega^{\wedge 1})$ is not pseudo-effective. 
Further, by (3), $\left( X, \Omega^{\wedge 1} - \epsilon (\Omega-\Omega^{\wedge 1}) \right)$ 
is klt for some small $\epsilon >0$. 
Hence we may run a $(K_X+\Omega^{\wedge 1})$-MMP and ends with a Mori fiber space $f: X_{\ell} \to Z$:
\[
X=:X_0 \dashrightarrow X_1  \dashrightarrow \cdots \dashrightarrow X_{\ell}.
\]

Let $\Omega_{i}$ be the push-forward of $\Omega$ on $X_{i}$. 
Then $\Omega _{i}$ also satisfies the conditions (1)--(3). 
Further, since 
$\Omega _{\ell} -\Omega _{\ell} ^{\wedge 1} \sim _{\mathbb{R}} 
-(K_{X_{\ell}}+\Omega _{\ell} ^{\wedge 1})$
is ample over $Z$, it follows that 
$f(\mathrm{Supp}\, \Omega _{\ell} ^{\ge 1}) = Z$. 
Hence, by Lemma \ref{lem:MFS}, the dual complex 
$\mathcal{D} ( (\Omega _{\ell} ^{\wedge 1}) ^{=1})$ is contractible. 

Let $R_i$ be the extremal ray of $\overline{\mathrm{NE}} (X_i)$ corresponding to the step of 
the MMP $X_i \dashrightarrow X_{i+1}$. 
Since 
$- (\Omega _{i} -\Omega _{i} ^{\wedge 1}) \cdot R_i < 0$, 
it follows that some component $D_i$ of 
$\mathrm{Supp} (\Omega _{i} ^{> 1})$\ ($= \mathrm{Supp} \left( (\Omega ^{\wedge 1}) ^{=1} \right)$) satisfies 
$D_i \cdot R_i > 0$. 
By Theorem \ref{thm:inv_mmp}, $\mathcal{D} ( (\Omega _{i} ^{\wedge 1}) ^{=1})$ and 
$\mathcal{D} ( (\Omega _{i+1} ^{\wedge 1}) ^{=1})$ are homotopy equivalent. 
Hence, $\mathcal{D}(\Omega ^{\ge 1}) = \mathcal{D} ( (\Omega ^{\wedge 1}) ^{=1})$ is contractible. 
\end{proof}

\begin{prop}\label{prop:ample}
Let $(X, \Delta)$ be a $\mathbb{Q}$-factorial projective pair over $k$ with the condition $(\star)$. 
Suppose that $-(K_X + \Delta)$ is ample. 
Then for any dlt blow-up $g: (Y, \Delta _Y) \to (X, \Delta)$, 
the dual complex $\mathcal{D}(\Delta _Y ^{\ge 1})$ is contractible, 
where we define $\Delta _Y$ by $K_Y + \Delta _{Y} = g^* (K_X + \Delta)$.
\end{prop}

\begin{proof}
By Proposition \ref{thm:scc2}, 
it suffices to find an effective $\mathbb{R}$-divisor $\Omega _Y$ on $Y$ such that 
\begin{enumerate}
\item[(1)] $(Y, \Omega _Y ^{\wedge 1})$ is dlt, 
\item[(2)] $K_Y + \Omega _Y \sim _{\mathbb{R}} 0$, 
\item[(3)] $\mathrm{Supp} (\Omega _Y^{>1}) = \mathrm{Supp} (\Omega _Y ^{\ge 1})$, and
\item[(4)] $\mathrm{Supp} (\Omega _Y ^{\ge 1}) = \mathrm{Supp} (\Delta _Y ^{\ge 1})$. 
\end{enumerate}

Since $X$ is $\mathbb{Q}$-factorial, 
there exists an effective $\mathbb{R}$-divisor $F$ on $Y$ such that $-F$ is $g$-ample and 
$\mathrm{Supp}\, F = \mathrm{Excep} (g)$. 
Since $-(K_Y + \Delta _Y)$ is the pullback of an ample $\mathbb{R}$-divisor 
$-(K_X + \Delta)$ on $X$, 
it follows that $-(K_Y + \Delta _Y)- \epsilon F$ is ample for any sufficiently small $\epsilon > 0$. 

Note that $\mathrm{Supp}\, F \subset \mathrm{Supp} (\Delta _Y ^{\ge 1})$. 
Thus, we can find an effective $\mathbb{R}$-divisor $B$ on $Y$ 
such that 
\begin{itemize}
\item $B \ge \epsilon F$, 
\item $-(K_Y + \Delta _Y) - B$ is still ample, and 
\item $\mathrm{Supp}\, B = \mathrm{Supp} (\Delta _Y ^{\ge 1})$. 
\end{itemize}
Then there exists an effective $\mathbb{R}$-divisor $A$ on $Y$ 
such that 
\begin{itemize}
\item $A \sim_{\mathbb{R}} -(K_Y + \Delta _Y)- B$, and 
\item $(Y, \Delta _Y ^{\wedge 1} + 2A)$ is dlt (cf.\ \cite[Lemma 2.8]{NT} in positive characteristic case). 
\end{itemize}
In particular, it follows that 
\begin{itemize}
\item $(Y, \Delta _Y ^{\wedge 1} + A)$ is dlt, and 
\item $(\Delta _Y  + A) ^{\ge 1} = \Delta _Y ^{\ge 1}$. 
\end{itemize}

Set $\Omega _Y :=  \Delta _Y + A + B$. 
Then (2) holds.  Since
\begin{itemize}
\item $\Omega _Y ^{\wedge 1} = (\Delta _Y + A) ^{\wedge 1} = \Delta _Y ^{\wedge 1} + A$, 
\end{itemize}
(1) also holds. (3) and (4) hold by the way of taking $A$ and $B$. 
\end{proof}

\begin{lem}\label{lem:Qnefbig}
Let $(X, \Delta)$ be a $\mathbb{Q}$-factorial projective pair over $k$ with the condition $(\star)$. 
Assume that $-(K_X + \Delta)$ is nef and big. 
Then there exists a dlt blow-up $g: (Y, \Delta _Y) \to (X, \Delta)$
such that the dual complex $\mathcal{D}(\Delta _Y ^{\ge 1})$ is contractible, 
where we define $\Delta _Y$ by $K_Y + \Delta _{Y} = g^* (K_X + \Delta)$. 
\end{lem}

\begin{proof}
Since $-(K_X + \Delta)$ is nef and big, 
there exists an effective $\mathbb{R}$-divisor $E$ such that 
$-(K_X + \Delta) - \epsilon E$ is ample for any real number $\epsilon$ 
satisfying $0<\epsilon \le 1$. 
Let $h: W \to X$ be a log resolution of $(X, \Delta + E)$. 
For sufficiently small $\epsilon >0$, we can assume that 
\begin{itemize}
\item[(1)] For any $h$-exceptional prime divisor $F$ with $a_F (X, \Delta) > 0$, 
it still holds that $a_F(X, \Delta + \epsilon E) > 0$. 
\end{itemize}
Let $g: Y \to X$ be a dlt blow-up of $(X, \Delta + \epsilon E)$ 
such that the birational map $Y \dasharrow W$ does not contract any divisor on $Y$ (Remark \ref{rmk:dltbup} (2)). 
By Proposition \ref{prop:ample}, 
the dual complex $\mathcal{D}((\Delta _Y + \epsilon g^* E) ^{\ge 1})$ is contractible. 
Hence it is sufficient to check the following three conditions (the conditions (2) and (3) below mean that $g$ is also a dlt blow-up of $(X, \Delta)$): 
\begin{itemize}
\item[(2)] $a_F(X, \Delta) \le 0$ for any $g$-exceptional divisor $F$. 
\item[(3)] $(Y, \Delta _Y ^{\wedge 1})$ is dlt. 
\item[(4)] $\mathrm{Supp} \left( (\Delta _Y + \epsilon g^* E) ^{\ge 1} \right) = \mathrm{Supp} (\Delta _Y ^{\ge 1})$. 
\end{itemize}
Since any $g$-exceptional divisor is $h$-exceptional, the conditions (2) and (4) follow from (1). 
By (2), it follows that $\Delta _Y \ge 0$ and hence (3) is obvious. 
\end{proof}

\begin{thm}\label{thm:nefbig}
Let $(X, \Delta)$ be a projective pair over $k$ with the condition $(\star)$. 
Assume that $-(K_X + \Delta)$ is nef and big. 
Then there exists a dlt blow-up $g: (Y, \Delta _Y) \to (X, \Delta)$
such that the dual complex $\mathcal{D}(\Delta _Y ^{\ge 1})$ is contractible, 
where we define $\Delta _Y$ by $K_Y + \Delta _{Y} = g^* (K_X + \Delta)$. 
Further, we can take such $g$ with the additional condition (3) in Theorem \ref{thm:dltmodif}. 
\end{thm}

\begin{proof}
Let $h: (X', \Delta _{X'}) \to (X, \Delta)$ be a dlt blow-up of $(X, \Delta)$ 
with the condition (3) in Theorem \ref{thm:dltmodif}. 
Then, $X'$ is $\mathbb{Q}$-factorial and $-(K_{X'} + \Delta _{X'})$ is still nef and big. 

By Lemma \ref{lem:Qnefbig}, 
there exists a dlt blow-up $(X'', \Delta _{X''}) \to (X', \Delta _{X'})$ such that 
the dual complex $\mathcal{D}(\Delta _{X''} ^{\ge 1})$ is contractible. 
Since the composition $X'' \to X$ is again a dlt blow-up of $(X, \Delta)$ 
with condition (3) in Theorem \ref{thm:dltmodif}, we complete the proof. 
\end{proof}

In characteristic zero, we can conclude the following stronger statement. 

\begin{thm}\label{thm:nefbig0}
Let $(X, \Delta)$ be a projective pair over $k$ with the condition (i) in $(\star)$. 
Assume that $-(K_X + \Delta)$ is nef and big. 
Then for any dlt blow-up $g: (Y, \Delta _Y) \to (X, \Delta)$, 
the dual complex $\mathcal{D}(\Delta _Y ^{\ge 1})$ is contractible, 
where we define $\Delta _Y$ by $K_Y + \Delta _{Y} = g^* (K_X + \Delta)$.
\end{thm}
\begin{proof}
The assertion follows from Theorem \ref{thm:nefbig} and Proposition \ref{prop:indep}. 
\end{proof}

\section{Vanishing theorem on log canonical Fano varieties}

In this section, we prove a vanishing theorem of Witt vector cohomology of Fano varieties of Ambro-Fujino type 
(Theorem \ref{thm:WAFV}).

\subsection{Non-klt locus of log Fano three-folds}
In this subsection, 
we prove Lemma \ref{lem:rationality}, Proposition \ref{prop:normality}, and Proposition \ref{prop:tree} 
which will be used in the proof of Theorem \ref{thm:WAFV}. 

\begin{lem}\label{lem:rationality}
Let $k$ be an algebraic closed field of characteristic $p > 5$. 
Let $(X, \Delta)$ be a three-dimensional projective log canonical pair over $k$ with $-(K_X + \Delta)$ ample. 
Suppose that $\mathrm{Nklt}(X, \Delta)$ is pure dimension one. 
Then each irreducible component of $\mathrm{Nklt}(X, \Delta)$ is a rational curve. 
\end{lem}

\begin{proof}
Let $C_0$ be an irreducible component of $\mathrm{Nklt}(X, \Delta)$. 
Let $f: Y \to X$ be a dlt blow-up of $(X, \Delta)$ (Theorem \ref{thm:dltmodif}). 
We define $\Delta _Y$ by $K_Y + \Delta _Y = f^* (K_X + \Delta)$. 
Let $E \subset \mathrm{Supp}\, \Delta _Y ^{=1}$ be a component which dominates $C_0$. 
We define $\Delta _E$ by $K_{E} + \Delta _E = (K_Y + \Delta _Y)|_{E}$. 
Then $K_E + \Delta _E$ is not nef since $-(K_X + \Delta)$ is ample. 
By the cone theorem for surfaces (cf.\ \cite[Proposition 3.15]{Tan14}), 
there exists a $(K_E + \Delta _E)$-negative rational curve $B$. 
Since $(K_E + \Delta _E) \cdot B < 0$ and 
$-(K_E + \Delta _E)$ is the pulled back of a divisor on $C_0$, 
it follows that $B$ dominates $C_0$, which proves the rationality of $C_0$. 
\end{proof}

We prove two properties (Proposition \ref{prop:G1}, Proposition \ref{prop:G2}) on the dual complex 
$\mathcal{D}(\Delta _{Y} ^{\ge 1})$ of a dlt blow-up. 
They will be used in the proof of Proposition \ref{prop:normality}. 

First, we introduce some notations. 
When we write that $G$ is a $\Delta$-complex, we regard $G$ as the set of its simplices. 
Hence, when we write $S \in G$, then $S$ is a simplices of $G$. 
\begin{defi}
\begin{enumerate}
\item For a $\Delta$-complex $G$, we denote by $|G|$ the topological space of $G$. Hence, $|G| = \bigcup _{S \in G} S$ as set. 
\item For a $\Delta$-complex $G$ and its subset $G' \subset G$, we define the \textit{star} $\mathrm{st}(G', G)$ by 
\[
\mathrm{st}(G', G) = \{ S \in G \mid \text{$S \cap S' \not = \emptyset$ for some $S' \in G'$} \}. 
\]
\item For a $\Delta$-complex $G$, and $S, S' \in G$, we write $S < S'$ when $S$ is a face of $S'$. 
\end{enumerate}
\end{defi}

\begin{prop}\label{prop:G1}
Let $k$ be an algebraic closed field of characteristic $p > 5$. 
Let $(X, \Delta)$ be a three-dimensional projective log canonical pair over $k$ with $-(K_X + \Delta)$ nef and big. 
Suppose that $C := \mathrm{Nklt}(X, \Delta)$ is of pure dimension one. 
Let $f: (Y, \Delta_Y) \to (X, \Delta)$ be a dlt blow-up such that 
$\mathrm{Supp}\, \Delta _Y ^{\ge 1} = f^{-1}(C)$ and that $G := \mathcal{D}(\Delta _Y ^{\ge 1})$ is contractible. 
Let $C_0$ be an irreducible component of $C = \mathrm{Nklt}(X, \Delta)$. 

Let $G'$ be the subcomplex of $G$ which consists of the strata of $\Delta _Y ^{\ge 1}$ which dominate $C_0$. 
Let $U := |G| \setminus |G'|$ be the open subset of $|G|$, and 
let $V \subset |G|$ be a sufficiently small open neighborhood of $|G'|$. 
Then following hold. 
\begin{itemize}
\item[(1)] $G'$ is a connected subcomplex of $G$ of dimension at most one.  
\item[(2)] For each connected component $U'$ of $U$, it follows that $V \cap U'$ is also connected. 
\end{itemize}
\end{prop}
\begin{proof}
We prove (1). Only the connectedness of $G'$ is non-trivial. 
Let $q \in C_0$ be a general closed point. Since $q$ is general, we may assume that 
\begin{itemize}
\item for a strata $S$ of $\Delta _Y ^{\ge 1}$, $q \in f(S)$ holds if and only if $f(S) = C_0$ holds. 
\end{itemize}
Then the connectedness of the fiber $f^{-1}(q)$ shows that the connectedness of $G'$.

For (2), consider the Mayer--Vietoris sequence 
\[
H_1(|G| , \mathbb{Q}) \to H_0(U \cap V, \mathbb{Q}) \to 
H_0(U, \mathbb{Q}) \oplus H_0(V, \mathbb{Q}) \to 
H_0(|G| , \mathbb{Q}) \to 0. 
\]
Here, $H_1 (|G|, \mathbb{Q}) = 0$ follows because $|G|$ is contractible, 
$H_0(|G| , \mathbb{Q}) = H_0(V, \mathbb{Q}) = 0$ follows because $|G|$ and $V$ are connected. 
Therefore, it follows that $H_0(U \cap V, \mathbb{Q}) \cong H_0(U, \mathbb{Q})$, 
which proves the claim. 
\end{proof}


\begin{prop}\label{prop:G2}
Let $(X, \Delta), (Y, \Delta _Y), C, C_0, G, G', U, V$ be as in the Proposition \ref{prop:G1}. 
Let $U'$ be a connected component of $U$. 
Consider a pair $(B, S)$ with the following conditions:
\begin{itemize}
\item[(a)] $B \in \mathrm{st}(G', G) \setminus G'$ is an edge, and $S \in G'$ is a vertex. 
\item[(b)] $S < B$ and $B \cap U' \not = \emptyset$. 
\end{itemize}
For any two pairs $(B, S)$ and $(B', S')$ with the conditions (a) and (b) above, 
the assertion is that 
there exists a sequence of pairs 
\[
(B, S) = (B_0, S_0), (B_1, S_1), \ldots , (B_k, S_k)=(B', S')
\]
with $k \ge 0$ and the following conditions:
\begin{itemize}
\item[(c)] Each $(B_j, S_j)$ satisfies the conditions (a) and (b). 
\item[(d)] For each $0 \le i \le k-1$, 
the pairs $(B_i, S_i), (B_{i+1}, S_{i+1})$ satisfy one of the following conditions: 
\begin{itemize}
\item[(d-1)] $S_i = S_{i+1}$ holds, and $B_i, B_{i+1} < F$ for some $2$-simplex $F \in G$. 
\item[(d-2)] $S_i \not = S_{i+1}$ holds, and $S_i, S_{i+1} < E$ for some edge $E \in G'$. 
Further, $B_i, B_{i+1}, E < F$ for some $2$-simplex $F \in G$. 
\end{itemize}
\end{itemize}
\end{prop}

\vspace{3mm}

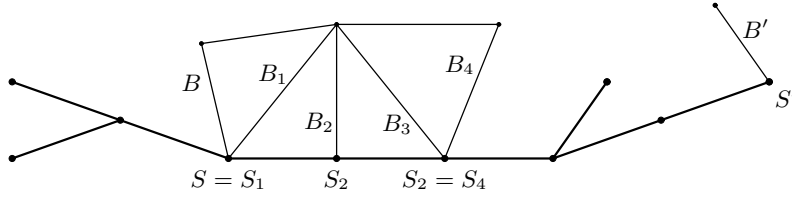
\begin{center}
\begin{footnotesize}
{\unitlength 0.1in%
\begin{picture}(39.4800,8.4500)(10.0000,-16.4500)%
%
\special{pn 4}%
\special{sh 1}%
\special{ar 1560 1400 16 16 0 6.2831853}%
\special{sh 1}%
\special{ar 2120 1600 16 16 0 6.2831853}%
\special{sh 1}%
\special{ar 1000 1600 16 16 0 6.2831853}%
\special{sh 1}%
\special{ar 1000 1200 16 16 0 6.2831853}%
\special{sh 1}%
\special{ar 2680 1600 16 16 0 6.2831853}%
\special{sh 1}%
\special{ar 3240 1600 16 16 0 6.2831853}%
\special{sh 1}%
\special{ar 3800 1600 16 16 0 6.2831853}%
\special{sh 1}%
\special{ar 4360 1400 16 16 0 6.2831853}%
\special{sh 1}%
\special{ar 4920 1200 16 16 0 6.2831853}%
\special{sh 1}%
\special{ar 4920 1200 16 16 0 6.2831853}%
%
\special{pn 4}%
\special{sh 1}%
\special{ar 4080 1200 16 16 0 6.2831853}%
\special{sh 1}%
\special{ar 4080 1200 16 16 0 6.2831853}%
%
\special{pn 13}%
\special{pa 1000 1200}%
\special{pa 1000 1200}%
\special{fp}%
\special{pa 1000 1200}%
\special{pa 1560 1400}%
\special{fp}%
\special{pa 1560 1400}%
\special{pa 1000 1600}%
\special{fp}%
\special{pa 1560 1400}%
\special{pa 2120 1600}%
\special{fp}%
\special{pa 2120 1600}%
\special{pa 2680 1600}%
\special{fp}%
\special{pa 2680 1600}%
\special{pa 3240 1600}%
\special{fp}%
\special{pa 3240 1600}%
\special{pa 3800 1600}%
\special{fp}%
\special{pa 3800 1600}%
\special{pa 4080 1200}%
\special{fp}%
\special{pa 3800 1600}%
\special{pa 4360 1400}%
\special{fp}%
\special{pa 4360 1400}%
\special{pa 4920 1200}%
\special{fp}%
%
\special{pn 4}%
\special{sh 1}%
\special{ar 1980 1000 10 10 0 6.2831853}%
\special{sh 1}%
\special{ar 2680 900 10 10 0 6.2831853}%
\special{sh 1}%
\special{ar 3520 900 10 10 0 6.2831853}%
\special{sh 1}%
\special{ar 4640 800 10 10 0 6.2831853}%
\special{sh 1}%
\special{ar 4640 800 10 10 0 6.2831853}%
%
\special{pn 8}%
\special{pa 1980 1000}%
\special{pa 2120 1600}%
\special{fp}%
\special{pa 2120 1600}%
\special{pa 2680 900}%
\special{fp}%
\special{pa 2680 900}%
\special{pa 1980 1000}%
\special{fp}%
\special{pa 3240 1600}%
\special{pa 2680 900}%
\special{fp}%
\special{pa 2680 900}%
\special{pa 2680 1600}%
\special{fp}%
%
\special{pn 8}%
\special{pa 2680 1600}%
\special{pa 2680 1600}%
\special{fp}%
%
\special{pn 8}%
\special{pa 3240 1600}%
\special{pa 3520 900}%
\special{fp}%
\special{pa 3520 900}%
\special{pa 2680 900}%
\special{fp}%
\special{pa 4920 1200}%
\special{pa 4640 800}%
\special{fp}%
\put(19.8000,-11.6000){\makebox(0,0)[rt]{$B$}}%
\put(24.2800,-12.1000){\makebox(0,0)[rb]{$B_1$}}%
\put(26.6600,-13.6000){\makebox(0,0)[rt]{$B_2$}}%
\put(30.7200,-13.7000){\makebox(0,0)[rt]{$B_3$}}%
\put(33.9400,-11.6000){\makebox(0,0)[rb]{$B_4$}}%
\put(21.2000,-17.1000){\makebox(0,0){$S=S_1$}}%
\put(26.8000,-17.1000){\makebox(0,0){$S_2$}}%
\put(32.4000,-17.1000){\makebox(0,0){$S_2=S_4$}}%
\put(49.4800,-12.3000){\makebox(0,0)[lt]{$S'$}}%
\put(47.8000,-9.7000){\makebox(0,0)[lb]{$B'$}}%
\end{picture}}%
\end{footnotesize}
\end{center}

\vspace{5mm}

\begin{proof}
Note that 
$U' \cap V$ is a connected component of $U \cap V$ by Proposition \ref{prop:G1} and 
that $U \cap V \subset \bigcup _{A \in \, \mathrm{st}\, (G', G) \setminus G'} \mathrm{int}\, A$. 

We also note that giving a pair $(B, S)$ with conditions (a) and (b) is equivalent to 
giving $B ^{\circ}$ in the following set. 
\begin{align*}
\Gamma := 
\left\{ B^{\circ} \ \middle |
\begin{array}{l}
\text{$B^{\circ}$ is a connected component of $B \cap (U' \cap V)$} \\ 
\text{for some edge $B \in \mathrm{st}(G, G') \setminus G'$}
\end{array}\right \}
\end{align*}
Indeed, for a pair $(B, S)$ with conditions (a) and (b), there exists the unique connected component of $B \cap (U' \cap V)$ which is around $S$. 
Inversely, for $B^ {\circ}$ in the set above, the corresponding $B$ and $S$ are uniquely determined. 

Let $B^{\circ}$ (resp.\ ${B'} ^{\circ}$) be the connected component of $B \cap (U' \cap V)$ 
(resp.\ $B' \cap (U' \cap V)$) which is around $S$ (resp.\ $S'$). 

Since $B^{\circ}, {B'} ^{\circ} \subset U' \cap V$ and $U' \cap V$ is connected, 
we can take the following sequence 
\[
B^{(0)} := B^{\circ}, F^{(0)}, B^{(1)}, F^{(1)}, \ldots , B^{(k-1)}, F^{(k-1)}, B^{(k)}:= {B'} ^{\circ}
\]
with the following conditions:
\begin{itemize}
\item Each $B^{(i)}$ is a connected component of $B_i \cap (U' \cap V)$ for some edge $B_i \in \mathrm{st}(G', G) \setminus G'$ 
(equivalently $B^{(i)} \in \Gamma$). 
\item Each $F^{(i)}$ is a connected component of $F_i \cap (U' \cap V)$ for some $2$-simplex $F_i \in \mathrm{st}(G', G) \setminus G'$. 
\item $B^{(i)}, B^{(i+1)} \subset F^{(i)}$
\end{itemize}
Possibly passing to a subsequence, we may also assume that 
\begin{itemize}
\item $B^{(i)} \not = B^{(i+1)}$ for each $0 \le i \le k-1$. 
\end{itemize}
We denote by $S_i$ the unique vertex of $B_i$ which is around $B^{(i)}$. Obviously, $S_i$ is a vertex of $G'$. 

For each $i$, there are two possibilities. 
\begin{itemize}
\item[(e-1)] The connected component of $F_i \cap |G'|$ which is around $F^{(i)}$ is zero dimensional. 
\item[(e-2)] The connected component of $F_i \cap |G'|$ which is around $F^{(i)}$ is one dimensional.
\end{itemize}

\vspace{3mm}

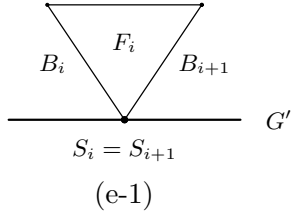
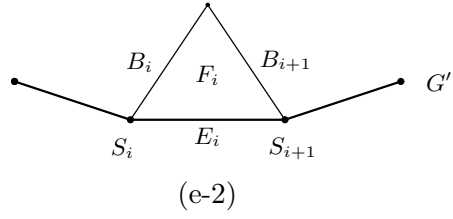
\begin{center}
\begin{footnotesize}
{\unitlength 0.1in%
\begin{picture}(42.9500,9.3500)(14.0000,-21.3500)%
%
\special{pn 4}%
\special{sh 1}%
\special{ar 2000 1800 16 16 0 6.2831853}%
\special{sh 1}%
\special{ar 2000 1800 16 16 0 6.2831853}%
%
\special{pn 13}%
\special{pa 2000 1800}%
\special{pa 1400 1800}%
\special{fp}%
\special{pa 2000 1800}%
\special{pa 2600 1800}%
\special{fp}%
%
\special{pn 4}%
\special{sh 1}%
\special{ar 1600 1200 8 8 0 6.2831853}%
\special{sh 1}%
\special{ar 2400 1200 8 8 0 6.2831853}%
\special{sh 1}%
\special{ar 2400 1200 8 8 0 6.2831853}%
%
\special{pn 8}%
\special{pa 2400 1200}%
\special{pa 2000 1800}%
\special{fp}%
\special{pa 2000 1800}%
\special{pa 1600 1200}%
\special{fp}%
\special{pa 1600 1200}%
\special{pa 2400 1200}%
\special{fp}%
%
\special{pn 4}%
\special{sh 1}%
\special{ar 4200 1800 16 16 0 6.2831853}%
\special{sh 1}%
\special{ar 5000 1800 16 16 0 6.2831853}%
\special{sh 1}%
\special{ar 5600 1600 16 16 0 6.2831853}%
\special{sh 1}%
\special{ar 3600 1600 16 16 0 6.2831853}%
\special{sh 1}%
\special{ar 3600 1600 16 16 0 6.2831853}%
%
\special{pn 13}%
\special{pa 3600 1600}%
\special{pa 4200 1800}%
\special{fp}%
\special{pa 4200 1800}%
\special{pa 5000 1800}%
\special{fp}%
\special{pa 5000 1800}%
\special{pa 5600 1600}%
\special{fp}%
%
\special{pn 4}%
\special{sh 1}%
\special{ar 4600 1200 8 8 0 6.2831853}%
\special{sh 1}%
\special{ar 4600 1200 8 8 0 6.2831853}%
%
\special{pn 8}%
\special{pa 4600 1200}%
\special{pa 4600 1200}%
\special{fp}%
\special{pa 4200 1800}%
\special{pa 4600 1200}%
\special{fp}%
\special{pa 4600 1200}%
\special{pa 5000 1800}%
\special{fp}%
\begin{normalsize}
\put(46.0000,-22.0000){\makebox(0,0){(e-2)}}%
\put(20.0000,-22.0000){\makebox(0,0){(e-1)}}%
\end{normalsize}
\put(20.0000,-19.6000){\makebox(0,0){$S_i = S_{i+1}$}}%
\put(28.0000,-18.0000){\makebox(0,0){$G'$}}%
\put(58.0000,-16.0000){\makebox(0,0){$G'$}}%
\put(20.0000,-14.0000){\makebox(0,0){$F_i$}}%
\put(17.0000,-14.6000){\makebox(0,0)[rt]{$B_i$}}%
\put(22.8000,-14.6000){\makebox(0,0)[lt]{$B_{i+1}$}}%
\put(46.0000,-15.7000){\makebox(0,0){$F_i$}}%
\put(43.2000,-15.5000){\makebox(0,0)[rb]{$B_i$}}%
\put(48.7000,-15.5000){\makebox(0,0)[lb]{$B_{i+1}$}}%
\put(41.6000,-19.5000){\makebox(0,0){$S_i$}}%
\put(50.4000,-19.5000){\makebox(0,0){$S_{i+1}$}}%
\put(46.0000,-18.9000){\makebox(0,0){$E_i$}}%
\end{picture}}%
\end{footnotesize}
\end{center}

\vspace{5mm}

Suppose $i$ satisfies (e-1). 
Then $S_i = S_{i+1}$ and $S_i$ is the common vertex of $B_i$ and $B_{i+1}$. 
Therefore $(B_i, S_i)$ and $(B_{i+1}, S_{i+1})$ satisfy (d-1). 

Suppose $i$ satisfies (e-2). 
Then there exists an edge $E_i \in G'$ such that $E_i, B_i, B_{i+1} < F_i$. 
Then $S_i$ (resp.\ $S_{i+1}$) is the common vertex of $E_i$ and $B_i$ (resp.\ $E_i$ and $B_{i+1}$). 
Then $(B_i, S_i)$ and $(B_{i+1}, S_{i+1})$ satisfy (d-2).
\end{proof}

\begin{prop}\label{prop:normality}
Let $k$ be an algebraic closed field of characteristic $p > 5$. 
Let $(X, \Delta)$ be a three-dimensional projective log pair over $k$ with $-(K_X + \Delta)$ nef and big. 
Suppose that $\mathrm{Nklt}(X, \Delta)$ is of pure dimension one. 
Then for each irreducible component $C_0$ of $\mathrm{Nklt}(X, \Delta)$, 
its normalization $\overline{C_0} \to C_0$ is a universal homeomorphism. 
\end{prop}

\begin{proof}
Set $C := \mathrm{Nklt}(X, \Delta)$. 
By contradiction, suppose that $p \in C_0$ is a singular point such that 
the normalization $\overline{C_0} \to C_0$ is not a universal homeomorphism around $p$. 
Let $p^{(1)}, \ldots , p^{(m)}$ be the inverse image of $p$. 
By the assumption, $m \ge 2$. 

Let $f: (Y, \Delta_Y) \to (X, \Delta)$ be a dlt blow-up such that 
$\mathrm{Supp}\, \Delta _Y ^{\ge 1} = f^{-1}(C)$ (Theorem \ref{thm:dltmodif}). 
Let $G = \mathcal{D}(\Delta _Y ^{\ge 1})$. 
We may assume that $G$ is contractible by Theorem \ref{thm:nefbig}. 

We define $\{ S_i \}_{i \in I}$ and $\{ T_j \}_{j \in J}$ as follows:
\begin{itemize}
\item Let $\{ S_i \}_{i \in I}$ be the set of the irreducible components of $\Delta _Y ^{\ge 1}$ which dominate $C_0$. 
\item Let $\{ T_j \}_{j \in J}$ be the set of the irreducible components $T_j$ of $\Delta _Y ^{\ge 1}$ which do not dominate $C_0$ but $p \in f(T_j)$. 
\end{itemize}

For each $S \in \{S_i \}_{i \in I}$, since $S$ is normal and dominates $C_0$, 
$S \to C_0$ factors through $S \to \overline{C_0} \to C_0$. 
We denote by $S_{p^{(k)}}$ the fiber $S \to \overline{C_0}$ over $p^{(k)}$. 
Obviously, we have 
\begin{itemize}
\item[(1)] $S_{p^{(k)}} \cap S_{p^{(j)}} = \emptyset$ holds for each $k \not = j$. 
\end{itemize}

On the other hand, $f^{-1}(p)$ is connected. 
Since $f^{-1}(p) \cap T \not = \emptyset$ for each $T \in \{ T_j \}_{j \in J}$, 
\[
\left( \bigcup_{i \in I,\ k} S_{i, p^{(k)}} \right) \cup \left( \bigcup _{j \in J} T_j \right)
= 
f^{-1}(p) \cup \bigcup _{j \in J} T_j
\]
is also connected. 
Hence, by (1), possibly changing the indices of $p^{(1)}, \ldots , p^{(m)}$, we can conclude the following: 
\begin{itemize}
\item[(2)] There exist $S, S' \in \{ S_i \} _{i \in I}$ and a sequence $T_1, \ldots , T_{\ell} \in \{ T_j \}_{j \in J}$ 
with $\ell \ge 0$ 
such that
\[
S_{p^{(1)}} \cap T_1 \not = \emptyset,\ T_1 \cap T_2 \not = \emptyset, \ \ldots ,\ T_{\ell - 1} \cap T_{\ell} \not = \emptyset,\ 
T_{\ell} \cap S'_{p^{(2)}} \not = \emptyset. 
\]
\end{itemize}
We shall rephrase (2) into a combinatorial condition (STEP 1) and lead a contradiction (STEP 2, 3). 

For STEP 1, we introduce some notations which are same as in Proposition \ref{prop:G1}. 
Let $G'$ be the subcomplex of $G$ which consists of the strata of $\Delta _Y ^{\ge 1}$ which dominate $C_0$. 
Let $U := |G| \setminus |G'|$ be the open subset of $|G|$, and 
let $V$ be a sufficiently small open neighborhood of $G'$. 

\vspace{2mm}

\noindent
\textbf{STEP 1. }\ 
In this step, we prove the following statement from the condition (2). 
\begin{itemize}
\item[(3)] There exist a connected component $U'$ of $U$, 
and pairs $(B, S)$ and $(B', S')$ with the condition (a) and (b) in Proposition \ref{prop:G2}
such that $S_{p^{(1)}} \cap B \not= \emptyset$ and $S' _{p^{(2)}} \cap B' \not = \emptyset$ hold. 
Here we allow $B = B'$ and $S=S'$. 
\end{itemize}

Suppose $\ell = 0$ in (2), that is, $S_{p^{(1)}} \cap S' _{p^{(2)}} \not = \emptyset$. 
Then, we can take a one dimensional stratum $B \subset S \cap S'$ such that 
$S_{p^{(1)}} \cap S' _{p^{(2)}} \cap B \not= \emptyset$ holds. 
If $B$ dominates $C_0$, then $B \to C_0$ also factors through $\overline{C_0}$ since $B$ is normal. 
We define $B_{p^{(1)}}$ and $B_{p^{(2)}}$ as the fibers of $B \to \overline{C_0}$ over $p^{(1)}$ and $p^{(2)}$. 
Then $S_{p^{(1)}} \cap B = B_{p^{(1)}}$ and 
$S' _{p^{(2)}} \cap B = B_{p^{(2)}}$ hold, and they contradict the fact that $B_{p^{(1)}} \cap B_{p^{(2)}} = \emptyset$. 
Therefore, $B$ does not dominate $C_0$, that is, $B \not \in G'$. 
Hence the condition (3) holds when we set $B' = B$. 

Suppose $\ell \ge 1$ in (2). Then 
we can take a one dimensional stratum $B \subset S \cap T_1$ (resp.\ $B' \subset T_{\ell} \cap S'$) such that 
$S_{p^{(1)}} \cap B \not =  \emptyset$ (resp.\ $S' _{p^{(2)}} \cap B' \not = \emptyset$). 
Since $T_1$ and $T_{\ell}$ do not dominate $C_0$, neither do $B$ and $B'$, hence $B, B' \not \in G'$. 
Moreover, $T_1 \cup \cdots \cup T_{\ell}$ is connected, $B$ and $B'$ have intersection with a same connected component $U'$ of $U$. 

We have proved (3). Here, we note that the condition $S_{p^{(1)}} \cap B \not= \emptyset$ 
(resp.\ $S' _{p^{(2)}} \cap B' \not = \emptyset$) in (3) actually implies that 
$B \subset S_{p^{(1)}}$ (resp.\ $B' \subset S' _{p^{(2)}}$). 
Indeed, it follows from the facts that $B \subset S$ (resp. $B' \subset S'$) and 
that $S$ (resp.\ $S'$) dominates $C_0$, but $B$ (resp.\ $B'$) does not dominate $C_0$. 
Therefore we conclude the following. 
\begin{itemize}
\item[(4)] There exist a connected component $U'$ of $U$, 
and pairs $(B, S)$ and $(B, S')$ with the condition (a) and (b) in Proposition \ref{prop:G2} such that 
$B \subset S_{p^{(1)}}$ and $B' \subset S' _{p^{(2)}}$ hold. 
Here we allow $B = B'$ and $S=S'$. 
\end{itemize}

\vspace{2mm}
\noindent
\textbf{STEP 2.}\ 
Let $(B, S), (B', S')$ be pairs which satisfy (a) and (b) in Proposition \ref{prop:G2}. 
Suppose that one of the following conditions holds: 
\begin{itemize}
\item[(d-1)] $S = S'$ holds, and $B, B' < F$ for some $2$-simplex $F \in G$. 
\item[(d-2)] $S \not = S'$ holds, and $S, S' < E$ for some edge $E \in G'$. 
Further, $B, B', E < F$ for some $2$-simplex $F \in G$. 
\end{itemize}
In this step, we prove the condition $B \subset S_{p^{(1)}}$ implies $B' \subset S' _{p^{(1)}}$. 

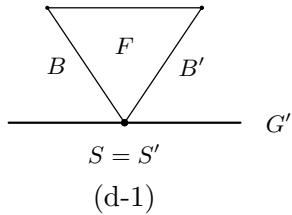
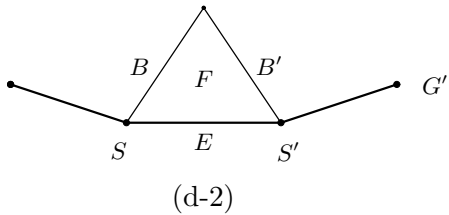
\begin{center}
\begin{footnotesize}
{\unitlength 0.1in%
\begin{picture}(42.9500,9.3500)(14.0000,-21.3500)%
%
\special{pn 4}%
\special{sh 1}%
\special{ar 2000 1800 16 16 0 6.2831853}%
\special{sh 1}%
\special{ar 2000 1800 16 16 0 6.2831853}%
%
\special{pn 13}%
\special{pa 2000 1800}%
\special{pa 1400 1800}%
\special{fp}%
\special{pa 2000 1800}%
\special{pa 2600 1800}%
\special{fp}%
%
\special{pn 4}%
\special{sh 1}%
\special{ar 1600 1200 8 8 0 6.2831853}%
\special{sh 1}%
\special{ar 2400 1200 8 8 0 6.2831853}%
\special{sh 1}%
\special{ar 2400 1200 8 8 0 6.2831853}%
%
\special{pn 8}%
\special{pa 2400 1200}%
\special{pa 2000 1800}%
\special{fp}%
\special{pa 2000 1800}%
\special{pa 1600 1200}%
\special{fp}%
\special{pa 1600 1200}%
\special{pa 2400 1200}%
\special{fp}%
%
\special{pn 4}%
\special{sh 1}%
\special{ar 4200 1800 16 16 0 6.2831853}%
\special{sh 1}%
\special{ar 5000 1800 16 16 0 6.2831853}%
\special{sh 1}%
\special{ar 5600 1600 16 16 0 6.2831853}%
\special{sh 1}%
\special{ar 3600 1600 16 16 0 6.2831853}%
\special{sh 1}%
\special{ar 3600 1600 16 16 0 6.2831853}%
%
\special{pn 13}%
\special{pa 3600 1600}%
\special{pa 4200 1800}%
\special{fp}%
\special{pa 4200 1800}%
\special{pa 5000 1800}%
\special{fp}%
\special{pa 5000 1800}%
\special{pa 5600 1600}%
\special{fp}%
%
\special{pn 4}%
\special{sh 1}%
\special{ar 4600 1200 8 8 0 6.2831853}%
\special{sh 1}%
\special{ar 4600 1200 8 8 0 6.2831853}%
%
\special{pn 8}%
\special{pa 4600 1200}%
\special{pa 4600 1200}%
\special{fp}%
\special{pa 4200 1800}%
\special{pa 4600 1200}%
\special{fp}%
\special{pa 4600 1200}%
\special{pa 5000 1800}%
\special{fp}%
\begin{normalsize}
\put(46.0000,-22.0000){\makebox(0,0){(d-2)}}%
\put(20.0000,-22.0000){\makebox(0,0){(d-1)}}%
\end{normalsize}
\put(20.0000,-19.6000){\makebox(0,0){$S = S'$}}%
\put(28.0000,-18.0000){\makebox(0,0){$G'$}}%
\put(58.0000,-16.0000){\makebox(0,0){$G'$}}%
\put(20.0000,-14.0000){\makebox(0,0){$F$}}%
\put(17.0000,-14.6000){\makebox(0,0)[rt]{$B$}}%
\put(22.8000,-14.6000){\makebox(0,0)[lt]{$B'$}}%
\put(46.0000,-15.7000){\makebox(0,0){$F$}}%
\put(43.2000,-15.5000){\makebox(0,0)[rb]{$B$}}%
\put(48.7000,-15.5000){\makebox(0,0)[lb]{$B'$}}%
\put(41.6000,-19.5000){\makebox(0,0){$S$}}%
\put(50.4000,-19.5000){\makebox(0,0){$S'$}}%
\put(46.0000,-19.0000){\makebox(0,0){$E$}}%
\end{picture}}%
\end{footnotesize}
\end{center}

\vspace{5mm}

Suppose (d-1). $S=S'$ in this case. Since $B, B' < F$ for some $2$-simplex $F \in G$, 
it follows that $B \cap B' \not= \emptyset$. 
Since $B \subset S_{p^{(1)}}$, it holds that $B' \cap S_{p^{(1)}} \not = \emptyset$. 
Then $B' \subset S_{p^{(1)}}$ holds because of 
the facts that $B' \subset S$ and 
that $S$ dominates $C_0$, but $B'$ does not dominate $C_0$.

Suppose (d-2). 
Since $B, B', E < F$ for some $2$-simplex $F \in G$, 
$B \cap B' \cap E \not = \emptyset$ holds. 
Since $B \subset S_{p^{(1)}}$, 
it follows that 
\[
S_{p^{(1)}} \cap B' \cap E \not = \emptyset. 
\]
Here we note that $E$ dominates $C_0$ and hence $E \to C_0$ factors through $\overline{C_0}$ because $E$ is normal. 
We denote by $E_{p^{(1)}}$ the fiber of $E \to \overline{C_0}$ over $p^{(1)}$. 
Since $E \subset S, S'$, it holds that 
\[
S_{p^{(1)}} \cap E = E_{p^{(1)}} \subset S' _{p^{(1)}}. 
\]
Hence we obtain that $B' \cap S' _{p^{(1)}} \not = \emptyset$. 
It implies that $B' \subset S' _{p^{(1)}}$ by the same reason as before. 

\vspace{2mm}
\noindent
\textbf{STEP 3.}\ In this step, we assume the condition (4) in STEP 1, and lead a contradiction. 

Let $(B, S), (B', S')$ be the pairs in (4) in STEP 1. 
Then $(B, S), (B', S')$ satisfy the condition (a), (b) in Proposition \ref{prop:G2}. 
Hence by Proposition \ref{prop:G2}, 
we can take a sequence $(B, S) = (B_0, S_0), (B_1, S_1), \ldots , (B_k, S_k)=(B', S')$ as in Proposition \ref{prop:G2}. 

By STEP 2 and the assumption that $B \subset S_{p^{(1)}}$, 
it follows that $B_k \subset S_{k, p^{(1)}}$ for each $k$ by induction. 
Therefore $B' \subset S'_{p^{(1)}}$ holds and it contradicts the assumption that $B' \subset S' _{p^{(2)}}$ 
and the fact that $S'_{p^{(1)}} \cap S'_{p^{(2)}} = \emptyset$. 
\end{proof}

In order to state Proposition \ref{prop:tree}, we introduce a notation. 

\begin{defi}\label{def:tree}
Let $C$ be a scheme of finite type of pure dimension one over an algebraic closed field $k$. 
Let $C = C_1 \cup C_2 \cup \cdots \cup C_{\ell}$ be the irreducible decomposition. 
We define whether $C$ \textit{forms a tree} or not 
by induction on the number $\ell$ of the irreducible components. 

Any irreducible curve \textit{forms a tree}. 
We call that the union of irreducible curves $C = C_1 \cup C_2 \cup \cdots \cup C_{\ell}$ \textit{forms a tree}, 
if there exists $i$ such that $C' = \bigcup _{j \not = i} C_j$ forms a tree and 
$\# (C' \cap C_i) = 1$. 
\end{defi}

First, we prepare some notation in combinatorics. 

\begin{defi}
For a sequence $(a_1, \ldots , a_n)$ of length $n$, we define the operations (a), (b) as follows. 
Here we define $a_{n+1} := a_1$ and $a_{n+2} := a_2$ by convention. 
\begin{itemize}
\item[(a)] If $a_i = a_{i+1}$ for some $1 \le i \le n$, we remove $a_i$ and 
get a new sequence $(a_1, \ldots , a_{i-1}, a_{i+1}, \ldots , a_n)$ of length $n-1$. 
\item[(b)] If $a_i = a_{i+2}$ for some $1 \le i \le n$, then we remove $a_i$ and $a_{i+1}$, and 
get a new sequence $(a_1, \ldots , a_{i-1}, a_i = a_{i+2} ,a_{i+3},  \ldots , a_n)$ of length $n-2$. 
\end{itemize}
For a sequence $(a_1, \ldots , a_n)$ of length $n$, applying the operation (a) (resp.\ the operations (a) and (b)) repeatedly, 
we get a sequence $(b_1, \ldots , b_m)$ with the condition that $b_i \not = b_{i+1}$ for each $i$ 
(resp. the condition that $b_i \not = b_{i+1}$ and $b_i \not = b_{i+2}$). 
We call such $(b_1, \ldots , b_m)$ \textit{the (a)-reduction} (resp. \textit{(a,b)-reduction}) of 
$(a_1, \ldots , a_n)$. 

We say that a sequence $(a_1, \ldots, a_n)$ has the \textit{trivial (a)-reduction} 
(resp.\ \textit{trivial (a,b)-reduction}) if the (a)-reduction (resp.\ the (a,b)-reduction) has length $0$. 
We note here that the (a)-reduction of a sequence of length $1$ has length $0$ by definition. 
\end{defi}
\begin{defi}\label{defi:curve_seq}
Let $C = C_1 \cup C_2 \cup \cdots \cup C_{\ell}$ be in Definition \ref{def:tree}. 
We call that a sequence $(a_1, \ldots, a_n)$ is a \textit{cycle sequence} if the following conditions (1), (2) hold. 
\begin{itemize}
\item[(1)] Each $a_i$ is an irreducible component of $C$ or a closed point on $C$. 
\item[(2)] The (a)-reduction $(b_1, \ldots , b_m)$ of $(a_1, \ldots , a_n)$ satisfies the following conditions:
\begin{itemize}
\item[(2-1)] $\dim b_i \not = \dim b_{i+1}$ for each $1 \le i \le m$. We set $b_{m+1} = b_1$ by convention. 
\item[(2-2)] If $\dim b_i = 0$, then $b_i \in b_{i-1} \cap b_{i+1}$. We set $b_0 = b_m$ by convention.  
\end{itemize}
\end{itemize}
Moreover, we say that a cycle sequence $(a_1, \ldots , a_n)$ is \textit{trivial} 
if it has the trivial (a,b)-reduction. 
\end{defi}

\begin{ex}
Consider curves $C=C_1 \cup \cdots \cup C_5$ in the following figure. 

\vspace{3mm}

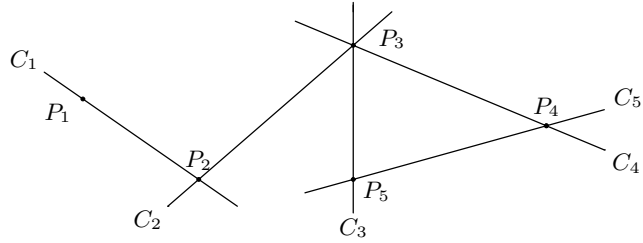
\begin{center}
\begin{footnotesize}
{\unitlength 0.1in%
\begin{picture}(32.9000,11.1100)(2.6000,-15.9100)%
%
\special{pn 4}%
\special{sh 1}%
\special{ar 800 984 10 10 0 6.2831853}%
\special{sh 1}%
\special{ar 1400 1404 10 10 0 6.2831853}%
\special{sh 1}%
\special{ar 2200 704 10 10 0 6.2831853}%
\special{sh 1}%
\special{ar 2200 1404 10 10 0 6.2831853}%
\special{sh 1}%
\special{ar 3200 1124 10 10 0 6.2831853}%
\special{sh 1}%
\special{ar 3200 1124 10 10 0 6.2831853}%
%
\special{pn 8}%
\special{pa 600 844}%
\special{pa 1600 1544}%
\special{fp}%
%
\special{pn 8}%
\special{pa 1240 1544}%
\special{pa 1240 1544}%
\special{fp}%
\special{pa 1240 1544}%
\special{pa 2400 529}%
\special{fp}%
%
\special{pn 8}%
\special{pa 2200 529}%
\special{pa 2200 494}%
\special{fp}%
\special{pa 2200 1579}%
\special{pa 2200 480}%
\special{fp}%
%
\special{pn 8}%
\special{pa 3500 1250}%
\special{pa 3500 1250}%
\special{fp}%
\special{pa 3500 1250}%
\special{pa 1900 578}%
\special{fp}%
\special{pa 1950 1474}%
\special{pa 3500 1040}%
\special{fp}%
\put(7.4000,-10.0500){\makebox(0,0)[rt]{$P_1$}}%
\put(14.0000,-13.0600){\makebox(0,0){$P_2$}}%
\put(5.7000,-8.3000){\makebox(0,0)[rb]{$C_1$}}%
\put(12.1000,-15.6500){\makebox(0,0)[rt]{$C_2$}}%
\put(22.0000,-16.5600){\makebox(0,0){$C_3$}}%
\put(35.4000,-12.7100){\makebox(0,0)[lt]{$C_4$}}%
\put(35.5000,-10.2600){\makebox(0,0)[lb]{$C_5$}}%
\put(23.3000,-7.1800){\makebox(0,0)[lb]{$P_3$}}%
\put(32.0000,-10.2600){\makebox(0,0){$P_4$}}%
\put(22.5000,-14.1800){\makebox(0,0)[lt]{$P_5$}}%
\end{picture}}%
\end{footnotesize}
\end{center}

\vspace{5mm}

\noindent
Set sequences $Q_1, Q_2$ as follows: 
\begin{align*}
Q_1&=(P_1, C_1, P_2, C_2, P_3, C_3, P_3, C_4, P_4, C_5, P_5, C_3, P_3, C_2, P_2, C_1), \\
Q_2&=(P_1, C_1, P_2, C_2, P_3, C_3, P_3, C_4, P_4, C_5, P_5, C_5, P_4, C_4, P_3, C_2, P_2, C_1). 
\end{align*}
Their (a)-reductions are themselves. 
These two sequences satisfy the condition (2-1), (2-2), hence they are cycle sequences. 
$Q_1$ is not trivial, but $Q_2$ is trivial. 
Indeed the (a,b)-reduction of $Q_1$ is $(C_3, P_3, C_4, P_4, C_5, P_5)$. 
\end{ex}

A typical example of cycle sequences we will see in the proof of Proposition \ref{prop:tree} is as follows. 

\begin{ex}\label{ex:edge_loop}
Let $k$ be an algebraic closed field of characteristic $p > 5$. 
Let $(X, \Delta)$ be a three-dimensional projective log pair over $k$ with $-(K_X + \Delta)$ nef and big. 
Suppose that $\mathrm{Nklt}(X, \Delta)$ is pure dimension one. 
Let $f: (Y, \Delta_Y) \to (X, \Delta)$ be a dlt blow-up, 
and let $G := \mathcal{D} (\Delta _Y ^{\ge 1})$ be the dual complex of $\Delta _Y ^{\ge 1}$. 
Then $G$ is a regular $\Delta$-complex by Proposition \ref{prop:regular}. 

For an edge $C$ in $G$, its two vertices $S$ and $S'$ are distinct because $G$ is regular. 
We denote by $C(S,S')$ the oriented $1$-cell corresponding to $C$ with initial point $S$ and final point $S'$. 

When we write
\[
P: S_1 \overset{C_1}{\longrightarrow} S_2 \overset{C_2}{\longrightarrow} \cdots 
\overset{C_{n-1}}{\longrightarrow} S_n  \overset{C_n}{\longrightarrow} S_{n+1}, 
\]
we assume the following condition. 
\begin{itemize}
\item For each $1 \le i \le n$, $S_i \not = S_{i+1}$ holds and $S_i$ and $S_{i+1}$ are the two vertices of $C_i$. 
\end{itemize}
We denote by $P$ the edge path obtained by joining the oriented $1$-cell $C_1(S_1, S_2), \ldots, C_n(S_n, S_{n+1})$. 
The edge path $P$ is called an \textit{edge loop} when $S_1 = S_{n+1}$. 

Let $P$ as above be an edge loop in $G$. 
Then $(f(S_1), f(C_1), \ldots , f(S_n), f(C_n))$ is a cycle sequence because $f(S_{n-1}) \supset f(C_n) \subset f(S_n)$ holds. 
We say that the sequence $(f(S_1), f(C_1), \ldots , f(S_n), f(C_n))$ is 
the \textit{image} of the edge loop $P$ for simplicity. 
\end{ex}

\begin{lem}\label{lem:cycle_seq}
Let $C = C_1 \cup C_2 \cup \cdots \cup C_{\ell}$ be in Definition \ref{def:tree}. 
Suppose that $C$ is connected but does not form a tree. 
Then there exists a non-trivial cycle sequence $(a_1, \ldots , a_n)$ such that $a_i$'s are distinct.
\end{lem}
\begin{proof}
Since $C$ does not form a tree, there exist irreducible components $B_1, \ldots, B_k$ of $C$ with the condition
\begin{itemize}
\item[(3)] $\# \left( B_i \cap (\bigcup _{j \not = i} B_j) \right) \ge 2$ for each $1 \le i \le k$. 
\end{itemize}
We set $b_1 = B_1$ and take a point $b_2$ in $B_1 \cap (\bigcup _{j \not = 1} B_j)$. 
Inductively, we set $b_{2i+1}$ and $b_{2i+2}$ for $i \ge 1$ as follows: 
\begin{itemize}
\item We take an arbitrary $B_m$ among $\{ B_1, \ldots , B_k \} \setminus \{ b_{2i-1} \}$ 
such that $b_{2i} \in b_{2i-1} \cap B_m$. 
\item We take an arbitrary point $p$ in 
$\left( B_{m} \cap (\bigcup _{j \not = m} B_j) \right) \setminus \{ b_{2i} \}$. 
\item We set $b_{2i+1} = B_m$ and $b_{2i+2} = p$. 
\end{itemize}
We can repeat this process by the condition (3). 
Since $\{ B_1, \ldots , B_k \}$ is a finite set, 
there exist $m_1, m_2$ with $m_2 \ge m_1 + 4$ such that 
\begin{itemize}
\item $b_{m_1} = b_{m_2}$ but $b_{m_1}, \ldots , b_{m_2 -1}$ are distinct. 
\end{itemize}
Then the sequence $(b_{m_1}, \ldots , b_{m_2 -1})$ 
satisfies the condition (1), (2) in Definition \ref{defi:curve_seq}. 
Since $b_{m_1}, \ldots , b_{m_2 -1}$ are distinct, the sequence has the non-trivial (a,b)-reduction. 
\end{proof}

\begin{prop}\label{prop:tree}
Let $k$ be an algebraic closed field of characteristic $p > 5$. 
Let $(X, \Delta)$ be a three-dimensional projective log pair over $k$ with $-(K_X + \Delta)$ nef and big. 
Suppose that $\mathrm{Nklt}(X, \Delta)$ is pure dimension one. 
Then $\mathrm{Nklt}(X, \Delta)$ forms a tree. 
\end{prop}
\begin{proof}
Note that $C := \mathrm{Nklt}(X, \Delta)$ is connected by \cite[Theorem 1.2]{NT}. 
Suppose that $C = \mathrm{Nklt}(X, \Delta)$ does not form a tree. 
Let $C = C_1 \cup \cdots \cup C_{\ell}$ be the irreducible decomposition. 

\vspace{2mm}

\noindent
\textbf{STEP 1. }\ 
Let $f: (Y, \Delta_Y) \to (X, \Delta)$ be a dlt blow-up such that $\mathrm{Supp}\, \Delta _Y ^{\ge 1} = f^{-1}(C)$. 
Let $G := \mathcal{D} (\Delta _Y ^{\ge 1})$ be the dual complex. 
We may assume that $G$ is contractible by Theorem \ref{thm:nefbig}. 

The assertion in this step is that there exists an edge loop (see Example \ref{ex:edge_loop} for the notation)
\[
S_1 \overset{C_1}{\longrightarrow} S_2 \overset{C_2}{\longrightarrow} \cdots 
\overset{C_{n-1}}{\longrightarrow} S_n  \overset{C_n}{\longrightarrow} S_1
\]
in $G$ such that 
\begin{itemize}
\item its image $(f(S_1), f(C_1), \ldots , f(S_n), f(C_n))$ is a non-trivial cycle sequence. 
\end{itemize}

By Lemma \ref{lem:cycle_seq}, there exists a non-trivial cycle sequence $(a_1, \ldots , a_n)$ such that $a_i$'s are distinct.
Since $a_i$'s are distinct, the (a)-reduction of this sequence is itself.  
Therefore, this sequence itself satisfies the conditions (2-1) and (2-2) in Definition \ref{defi:curve_seq}. 
We may assume that $\dim a_1 = 0$. 

Suppose that $i$ is odd. 

\vspace{2mm}

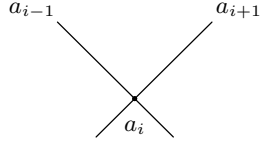
\begin{center}
\begin{footnotesize}
{\unitlength 0.1in%
\begin{picture}(12.9000,7.5000)(7.3000,-16.0000)%
%
\special{pn 4}%
\special{sh 1}%
\special{ar 1600 1400 10 10 0 6.2831853}%
\special{sh 1}%
\special{ar 1600 1400 10 10 0 6.2831853}%
%
\special{pn 8}%
\special{pa 1200 1000}%
\special{pa 1800 1600}%
\special{fp}%
\special{pa 1400 1600}%
\special{pa 2000 1000}%
\special{fp}%
\put(11.9000,-9.8000){\makebox(0,0)[rb]{$a_{i-1}$}}%
\put(20.2000,-9.8000){\makebox(0,0)[lb]{$a_{i+1}$}}%
\put(16.0000,-15.5000){\makebox(0,0){$a_i$}}%
\end{picture}}%
\end{footnotesize}
\end{center}

\vspace{2mm}

\noindent
Then, we can take an edge path $P_i$: 
\[
P_i: S^{(i)}_1 \overset{C^{(i)}_1}{\longrightarrow} S^{(i)}_2 \overset{C^{(i)}_2}{\longrightarrow} \cdots
 \overset{C^{(i)}_{m_i -1}}{\longrightarrow} S^{(i)}_{m_i}
\]
in $G$ with the following conditions: 
\begin{itemize}
\item[(4)] $f(S^{(i)}_1) = a_{i-1}$ and $f(S^{(i)} _{m_i}) = a_{i+1}$. 
\item[(5)] $a_i \in f(S^{(i)}_j)$ but $f(S^{(i)}_j) \not = a_{i-1}, a_{i+1}$ for $2 \le j \le m_i -1$. 
\end{itemize}
Such $P_i$ can be taken by the connectedness of 
the subcomplex of $G$ which consists of the simplices corresponding to the stratum $S$ of 
$\Delta _Y ^{\ge 1}$ which satisfies $a_i \in f(S)$.

Suppose that $i$ is even. 

\vspace{2mm}

\begin{center}
\begin{footnotesize}
{\unitlength 0.1in%
\begin{picture}(10.8000,1.7500)(11.7000,-10.3500)%
%
\special{pn 4}%
\special{sh 1}%
\special{ar 1400 1000 10 10 0 6.2831853}%
\special{sh 1}%
\special{ar 2000 1000 10 10 0 6.2831853}%
\special{sh 1}%
\special{ar 2000 1000 10 10 0 6.2831853}%
%
\special{pn 8}%
\special{pa 1180 1000}%
\special{pa 2220 1000}%
\special{fp}%
\put(22.5000,-9.9000){\makebox(0,0)[lb]{$a_i$}}%
\put(14.0000,-11.0000){\makebox(0,0){$a_{i-1}$}}%
\put(20.0000,-11.0000){\makebox(0,0){$a_{i+1}$}}%
\end{picture}}%
\end{footnotesize}
\end{center}

\vspace{3mm}

\noindent
Then, we can take an edge path $P_i$: 
\[
P_i: S^{(i)}_1 \overset{C^{(i)}_1}{\longrightarrow} S^{(i)}_2 \overset{C^{(i)}_2}{\longrightarrow} \cdots
 \overset{C^{(i)}_{m_i -1}}{\longrightarrow} S^{(i)}_{m_i}
\]
in $G$ with the following conditions:
\begin{itemize}
\item[(6)] $S_1 ^{(i)} = S_{m_{i-1}} ^{(i-1)}$ and 
$S_{m_i} ^{(i)} = S_1 ^{(i+1)}$ (here, $S_{m_{i-1}} ^{(i-1)}$ and $S_1 ^{(i+1)}$ were already taken). 
\item[(7)] $f(S^{(i)}_j) = a_i$ for $1 \le j \le m_i$ and 
$f(C^{(i)}_j) = a_i$ for $1 \le j \le m_i - 1$. 
\end{itemize}
Such $P_i$ can be taken by the connectedness of 
the subcomplex of $G$ which consists of the simplices corresponding to the stratum $S$ of 
$\Delta _Y ^{\ge 1}$ which satisfies $a_i = f(S)$ (Proposition \ref{prop:G1} (1)). 

Connecting the  edge paths $P_1, P_2, \ldots , P_n$, we get an edge loop $P$, 
which is possibly not simple. 

We prove that the image of $P$ is a non-trivial cycle sequence. 
Note that the image of any edge loop in $G$ is a cycle sequence (Example \ref{ex:edge_loop}). 
Therefore, it is sufficient to show that the (a,b)-reduction of the image of $P$ is non-trivial. 

For an even number $i$, it follows that 
\[
f(S^{(i)} _1) = f(C^{(i)} _1) = \cdots = 
f(C^{(i)} _{m_i -1}) =  f(S^{(i)} _{m_i}) = a_{i}
\]
by the condition (7). 
Hence, applying the operation (a) to the image of $P$ 
\[
\left( \ldots,  f(C_{m_{i-1} - 1}^{(i-1)}) , f(S_{m_{i-1}}^{(i-1)}) = f(S_1 ^{(i)}), 
\ldots , f(S_{m_i} ^{(i)}) = f(S_{1} ^{(i+1)}), f(C_{1} ^{(i+1)} ) , \ldots  \right), 
\]
it can be reduced to 
\[
\left( \ldots,  f(C_{m_{i-1} - 1}^{(i-1)}) , f(S_{m_{i-1}}^{(i-1)}) = a_i = f(S_{1} ^{(i+1)}), 
f(C_{1} ^{(i+1)} ) , \ldots  \right). 
\]

For an odd number $i$, we set
\[
b_1 ^{(i)} = f(S^{(i)} _1), \quad b_2^{(i)} = f(C^{(i)} _1), \quad \ldots , \quad 
b_{2(m_i-1)}^{(i)} = f(C^{(i)} _ {m_i -1}), \quad b^{(i)} _{2m_i -1} = f(S^{(i)} _{m_i}). 
\]
Then, they satisfy the following conditions: 
\begin{itemize}
\item $b_1 ^{(i)} = a_{i-1}$ and $b^{(i)} _{2m_i -1} = a_{i+1}$ by the condition (4). 
\item $a_i \in b_j ^{(i)}$ but $b_j ^{(i)} \not = a_{i-1}, a_{i+1}$  for $2 \le j \le 2m_i -2$ by the condition (5). 
\item If $b_j ^{(i)} \not = b_{j+1} ^{(i)}$, then either $b_j ^{(i)} = a_i$ or $b_{j+1} ^{(i)} = a_i$ 
(This is because $a_i$ is a point and 
we have an inclusion either $a_i \in b_j ^{(i)} \subset b_{j+1} ^{(i)}$ or $a_i \in b_{j+1} ^{(i)} \subset b_{j} ^{(i)}$). 
\end{itemize}
Hence, applying the operation (a) to the image of $P$
\[
\left( \ldots , f(C^{(i-1)} _{m_{i-1} -1}), 
b_1^{(i)}, \ldots, b_{2m_i -1} ^{(i)}, f(C^{(i+1)} _1) \ldots  \right), 
\]
it can be reduced to
\[
\left(  \ldots , f(C^{(i-1)} _{m_{i-1} -1} ), 
a_{i-1}, a_i, c_1, a_i, \cdots, a_i, c_{n_i} ,a_i, a_{i+1}, f(C^{(i+1)} _1),  \ldots \right), 
\]
for some $c_1, \ldots, c_{n_i}$. 
Applying the operation (b), it can be reduced to 
\[
\left( \ldots , f(C^{(i-1)} _{m_{i-1} -1}), 
f(S_{m_i}^{(i-1)})= a_{i-1}, a_i, a_{i+1}= f(S_{1}^{(i+1)}), f(C^{(i+1)} _1), \ldots  \right). 
\]
Therefore, the (a,b)-reduction of the image of $P$ is $(a_1, a_2, \ldots , a_n)$, 
which is non-trivial since $a_i$'s are distinct. 
We have proved that the image of $P$ is a non-trivial cycle sequence.

\vspace{2mm}

\noindent
\textbf{STEP 2. }\ 
Let $Q$ be an edge loop 
\[
Q: S'_1 \overset{C'_1}{\longrightarrow} S'_2 \overset{C'_2}{\longrightarrow} \cdots 
\overset{C'_{n-1}}{\longrightarrow} S'_n  \overset{C'_n}{\longrightarrow} S'_1
\]
in $G$. 
Suppose that $C'_i \not = C' _{i+1}$ there exist $i$ and a $2$-simplex $F$ in $G$ such that $C' _i, C' _{i+1} < F$. 
Let $C' <  F$ be the edge which is different from $C' _i$ and $C' _{i+1}$.

\vspace{1mm}

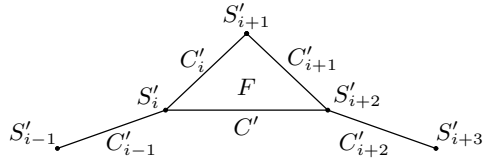
\begin{center}
\begin{footnotesize}
{\unitlength 0.1in%
\begin{picture}(24.4000,7.5500)(13.1200,-14.0000)%
%
\special{pn 4}%
\special{sh 1}%
\special{ar 1792 1400 10 10 0 6.2831853}%
\special{sh 1}%
\special{ar 2352 1200 10 10 0 6.2831853}%
\special{sh 1}%
\special{ar 3192 1200 10 10 0 6.2831853}%
\special{sh 1}%
\special{ar 3752 1400 10 10 0 6.2831853}%
\special{sh 1}%
\special{ar 2772 800 10 10 0 6.2831853}%
\special{sh 1}%
\special{ar 2772 800 10 10 0 6.2831853}%
%
\special{pn 8}%
\special{pa 1792 1400}%
\special{pa 2352 1200}%
\special{fp}%
\special{pa 2352 1200}%
\special{pa 3192 1200}%
\special{fp}%
\special{pa 3192 1200}%
\special{pa 3752 1400}%
\special{fp}%
\special{pa 3192 1200}%
\special{pa 2772 800}%
\special{fp}%
\special{pa 2772 800}%
\special{pa 2352 1200}%
\special{fp}%
\put(27.7200,-10.8000){\makebox(0,0){$F$}}%
\put(20.4400,-13.1000){\makebox(0,0)[lt]{$C_{i-1}'$}}%
\put(35.1400,-13.1000){\makebox(0,0)[rt]{$C_{i+2}'$}}%
\put(17.9200,-13.9000){\makebox(0,0)[rb]{$S_{i-1}'$}}%
\put(37.5200,-13.9000){\makebox(0,0)[lb]{$S_{i+3}'$}}%
\put(23.2400,-11.9000){\makebox(0,0)[rb]{$S_i'$}}%
\put(32.2000,-11.9000){\makebox(0,0)[lb]{$S_{i+2}'$}}%
\put(27.7200,-12.8000){\makebox(0,0){$C'$}}%
\put(25.6200,-10.0000){\makebox(0,0)[rb]{$C_i'$}}%
\put(29.8200,-10.0000){\makebox(0,0)[lb]{$C_{i+1}'$}}%
\put(27.7200,-7.1000){\makebox(0,0){$S_{i+1}'$}}%
\end{picture}}%
\end{footnotesize}
\end{center}

\vspace{3mm}

\noindent
Then we have a new edge loop $Q'$: 
\[
Q': S'_1 \overset{C'_1}{\longrightarrow} S'_2 \overset{C'_2}{\longrightarrow} \cdots 
\overset{C'_{i-1}}{\longrightarrow} S'_i \overset{C'}{\longrightarrow} S'_{i+2} \overset{C'_{i+2}}{\longrightarrow}
\cdots 
\overset{C'_{n-1}}{\longrightarrow} S'_n  \overset{C'_n}{\longrightarrow} S'_1. 
\]

\noindent
We claim in this step that 
\begin{itemize}
\item the image of $Q$ and the image of $Q'$ have the same (a,b)-reduction. 
\end{itemize}

The image of $Q$ is 
\[
R: \left( f(S' _1), f(C' _1), \ldots , f(S' _i), f(C' _i), f(S' _{i+1}), f(C' _{i+1}), f(S' _{i+2}), \ldots , f(C'_n) \right), 
\]
and the image of $Q'$ is 
\[
R': \left( f(S' _1), f(C' _1), \ldots , f(S' _i), f(C'), f(S' _{i+2}), \ldots , f(C'_n) \right). 
\]
We have four cases. 
\begin{itemize}
\item[(i)] $f(S' _i) = f(S' _{i+1}) = f(S' _{i+2})$. 
\item[(ii)] $f(S'_i) = f(S' _{i+2}) \not = f(S'_{i+1})$. 
\item[(iii)] $f(S'_i) = f(S' _{i+1}) \not = f(S' _{i+2})$ or $f(S'_i) \not = f(S' _{i+1}) = f(S' _{i+2})$. 
\item[(iv)] $f(S' _i), f(S' _{i+1}), f(S' _{i+2})$ are distinct. 
\end{itemize}

Suppose (i). 
Applying the operation (b) twice to $R$, and once to $R'$, we get the same sequence
\[
\left( f(S' _1), f(C' _1), \ldots , f(C'_{i-1}),f(S'_i) = f(S'_{i+2}) ,f(C'_{i+2}), \ldots , f(C'_n) \right).
\]

Suppose (ii). Since $f(C'_i) \subset f(S' _i) \cap f(S' _{i+1})$ and $f(S' _i) \not = f(S' _{i+1})$, 
it follows that $\dim f(C'_i) = 0$. By the same reason, it follows that $\dim f(C'_{i+1}) = 0$. 
Since $f(F) \subset f(C' _i) \cap f(C' _{i+1})$, it follows that $f(C' _i) = f(F) = f(C' _{i+1})$. 
Applying the operation (b) twice to $R$, and once to $R'$, we get the same sequence
\[
\left( f(S' _1), f(C' _1), \ldots , f(C'_{i-1}),f(S'_i)=f(S'_{i+2}),f(C'_{i+2}), \ldots , f(C'_n) \right).
\]

Suppose (iii). 
We may assume that $f(S'_i) = f(S' _{i+1}) \not = f(S' _{i+2})$. 
By the same reason as in the case (ii), $f(C' _{i+1}) = f(F) = f(C')$. 
Applying the operation (b) once to $R$, we get the sequence $R'$
\[
\left( f(S' _1), f(C' _1), \ldots , f(C'_{i-1}),f(S'_i)=f(S'_{i+1}),f(C' _{i+1})=f(C'), 
f(S' _{i+2}), f(C'_{i+2}), \ldots , f(C'_n) \right).
\]

Suppose (iv). In this case, $f(C' _i) = f(C' _{i+1}) = f(C')$. 
Applying the operation (b) once to $R$, we get the sequence $R'$
\[
\left( f(S' _1), f(C' _1), \ldots , f(C'_{i-1}),f(S'_i), f(C' _i) = f(C' _{i+1}) = f(C'), f(S' _{i+2}), 
f(C'_{i+2}), \ldots , f(C'_n) \right).
\]

In any case, $R$ and $R'$ have the same (a,b)-reduction. 

\vspace{2mm}

\noindent
\textbf{STEP 3. }\ 
Let $Q$ be an edge loop 
\[
Q: S'_1 \overset{C'_1}{\longrightarrow} S'_2 \overset{C'_2}{\longrightarrow} \cdots 
\overset{C'_{i-1}}{\longrightarrow} S' _i \overset{C'_i}{\longrightarrow} S'_{i+1} \overset{C'_{i+1}}{\longrightarrow} S'_{i+2} 
\overset{C'_{i+2}}{\longrightarrow} \cdots
\overset{C'_{n-1}}{\longrightarrow} S'_n  \overset{C'_n}{\longrightarrow} S'_1
\]
in $G$. 
Suppose that there exist $i$ such that $C' _i = C' _{i+1}$. 
Then $S'_i = S'_{i+2}$ holds.

\vspace{2mm}

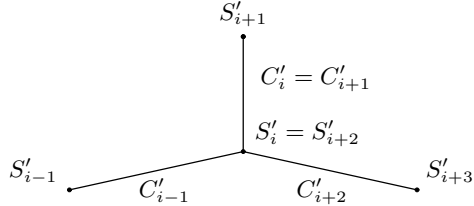
\begin{center}
\begin{footnotesize}
{\unitlength 0.1in%
\begin{picture}(23.8500,9.8500)(11.2000,-14.0000)%
%
\special{pn 4}%
\special{sh 1}%
\special{ar 2560 1200 10 10 0 6.2831853}%
\special{sh 1}%
\special{ar 1660 1400 10 10 0 6.2831853}%
\special{sh 1}%
\special{ar 3460 1400 10 10 0 6.2831853}%
\special{sh 1}%
\special{ar 2560 600 10 10 0 6.2831853}%
\special{sh 1}%
\special{ar 2560 600 10 10 0 6.2831853}%
%
\special{pn 8}%
\special{pa 2560 600}%
\special{pa 2560 1200}%
\special{fp}%
\special{pa 2560 1200}%
\special{pa 1660 1400}%
\special{fp}%
\special{pa 2560 1200}%
\special{pa 3460 1400}%
\special{fp}%
\put(25.6000,-4.8000){\makebox(0,0){$S_{i+1}'$}}%
\put(16.0000,-13.7000){\makebox(0,0)[rb]{$S_{i-1}'$}}%
\put(35.0500,-13.7000){\makebox(0,0)[lb]{$S_{i+3}'$}}%
\put(20.2000,-13.4000){\makebox(0,0)[lt]{$C_{i-1}'$}}%
\put(31.0000,-13.4000){\makebox(0,0)[rt]{$C_{i+2}'$}}%
\put(26.5000,-7.5000){\makebox(0,0)[lt]{$C_i' = C_{i+1}'$}}%
\put(26.2000,-11.7000){\makebox(0,0)[lb]{$S_i ' = S_{i+2}'$}}%
\end{picture}}%
\end{footnotesize}
\end{center}

\vspace{4mm}

\noindent
Then we have a new edge loop $Q'$: 
\[
Q': S'_1 \overset{C'_1}{\longrightarrow} S'_2 \overset{C'_2}{\longrightarrow} \cdots 
\overset{C'_{i-1}}{\longrightarrow} S'_i \overset{C'_{i+2}}{\longrightarrow} S'_{i+3} \overset{C'_{i+3}}{\longrightarrow}
\cdots 
\overset{C'_{n-1}}{\longrightarrow} S'_n  \overset{C'_n}{\longrightarrow} S'_1. 
\]

We claim in this step that 
\begin{itemize}
\item the image of $Q$ and the image of $Q'$ have the same (a,b)-reduction. 
\end{itemize}

The image of $Q$ is 
\[
R: \left( f(S' _1), f(C' _1), \ldots , f(S' _i), f(C' _i), f(S' _{i+1}), f(C' _{i+1}), f(S' _{i+2}),f(C' _{i+2}), \ldots , f(C'_n) \right), 
\]
and the image of $Q'$ is 
\[
R': \left( f(S' _1), f(C' _1), \ldots , f(S' _i), f(C'_{i+2}), \ldots , f(C'_n) \right). 
\]

Since $C' _i = C' _{i+1}$ and $S'_i = S'_{i+2}$, 
applying the operation (b) twice to $R$, we get $R'$. 
Hence $R$ and $R'$ have the same (a,b)-reduction. 

\vspace{2mm}

\noindent
\textbf{STEP 4. }\ 
By STEP 1, 
there exists an edge loop $Q$
\[
Q: S_1 \overset{C_1}{\longrightarrow} S_2 \overset{C_2}{\longrightarrow} \cdots 
\overset{C_{n-1}}{\longrightarrow} S_n  \overset{C_n}{\longrightarrow} S_1
\]
in $G$ such that 
\begin{itemize}
\item its image $(f(E_1), f(C_1), \ldots , f(E_n), f(C_n))$ is a non-trivial cycle sequence. 
\end{itemize}
Since $G$ is simply connected, applying the operation in STEP 2 and STEP 3 (and the reversing operation) repeatedly, 
we get a trivial path
\[
Q': S'
\]
for some vertex $S'$ \cite[Theorem 3.4.1]{Geo08}. 
The (a,b)-reduction of the image of $Q$ is not trivial but that of $Q'$ is trivial. 
This contradicts STEP 2 and STEP 3. 
\end{proof}

\subsection{Vanishing theorem of Witt vector cohomology of Ambro-Fujino type}

\begin{thm}\label{thm:WAFV}
Let $k$ be a perfect field of characteristic $p > 5$. 
Let $(X, \Delta)$ be a three-dimensional projective $\mathbb{Q}$-factorial log canonical pair over $k$ with $-(K_X + \Delta)$ ample. 
Then $H^i(X,W \mathcal{O}_{X, \mathbb{Q}}) = 0$ holds for $i > 0$. 
\end{thm}

\begin{proof}
By replacing $\Delta$ smaller, we may assume that $\dim \mathrm{Nklt}(X, \Delta) \le 1$. 

By the exact sequence
\[
0 \to WI_{\mathrm{Nklt}(X, \Delta), \mathbb{Q}} \to W \mathcal{O}_{X, \mathbb{Q}} \to 
W \mathcal{O}_{\mathrm{Nklt}(X, \Delta), \mathbb{Q}} \to 0, 
\]
and the Nadel type vanishing $H^i(X, WI_{\mathrm{Nklt}(X, \Delta), \mathbb{Q}}) = 0$ for $i >0$ 
(Theorem \ref{thm:WNV}), 
it is sufficient to show that 
\[
H^1 (\mathrm{Nklt}(X, \Delta), W \mathcal{O}_{\mathrm{Nklt}(X, \Delta), \mathbb{Q}}) = 0. 
\]
Here, we may assume that $k$ is algebraically closed by \cite[Lemma 2.15]{NT}.  
Since $\dim \mathrm{Nklt}(X, \Delta) \le 1$ and $\mathrm{Nklt}(X, \Delta)$ is connected (\cite[Theorem 1.2]{NT}), 
we may assume that $\mathrm{Nklt}(X, \Delta)$ is a union of curves. 

Let $C := \mathrm{Nklt}(X, \Delta) = C_1 \cup C_2 \cup \cdots \cup C_l$ be the irreducible decomposition. 
By Lemma \ref{lem:rationality}, Proposition \ref{prop:normality}, and Proposition \ref{prop:tree}, 
the curve $C$ satisfies the following conditions. 
\begin{enumerate}
\item[(1)] Each $C_i$ is a rational curve. 
\item[(2)] Each normalization of $C_i$ is a universal homeomorphism. 
\item[(3)] $C = C_1 \cup \ldots \cup C_l$ forms a tree (see Definition \ref{def:tree}). 
\end{enumerate}
Then $H^1 (C_i, W \mathcal{O}_{C_i, \mathbb{Q}}) = 0$ follows from (1) and (2) (cf.\ \cite[Lemma 2.21, 2.22]{GNT}). 
Hence the desired vanishing $H^1 (C, W \mathcal{O}_{C, \mathbb{Q}}) = 0$ follows from (3). 
\end{proof}

\subsection{Rational point formula}
As an application of Theorem \ref{thm:WAFV}, we obtain the following rational point formula. 

\begin{thm}\label{thm:RPF}
Let $k$ be a finite field of characteristic $p > 5$. 
Let $(X, \Delta)$ be a geometrically connected three-dimensional projective $\mathbb{Q}$-factorial log canonical pair over $k$ with $-(K_X + \Delta)$ ample. 
Then the number of the $k$-rational points on the non-klt locus on $(X, \Delta)$ satisfies 
\[
\# \mathrm{Nklt}(X, \Delta) (k) \equiv 1 \mod {|k|}.
\]
In particular, there exists a $k$-rational point on $\mathrm{Nklt}(X, \Delta)$. 
\end{thm}

\begin{proof}
Let $Z = \mathrm{Nklt}(X, \Delta)$ and let $I_Z$ be the corresponding coherent ideal sheaf. 
By Theorem \ref{thm:WNV} and Theorem \ref{thm:WAFV}, 
\[
H^i(X, WI_{Z, \mathbb{Q}}) = 0,\ \text{and}\ H^i(X,W \mathcal{O}_{X, \mathbb{Q}}) = 0
\]
hold for $i > 0$. 
By the exact sequence 
\[
0 \to WI_{Z} \to W \mathcal{O}_{X, \mathbb{Q}} \to 
W \mathcal{O}_{Z, \mathbb{Q}} \to 0, 
\]
$H^i(Z, W \mathcal{O}_{Z, \mathbb{Q}}) = 0$ holds for $i >0$. 
By \cite[Proposition 6.9 (i)]{BBE07}, it follows that $\# Z (k) \equiv 1 \pmod {|k|}$.
\end{proof}


\end{document}